\documentclass[a4paper,english,fontsize=11pt,parskip=half]{scrartcl} 
\usepackage{babel}
\usepackage[utf8]{inputenc}
\usepackage[T1]{fontenc}
\usepackage[a4paper,left=20mm,right=20mm,top=30mm,bottom=30mm]{geometry}
\usepackage{amsmath}
\usepackage{amsthm}
\usepackage{amssymb}
\usepackage{enumerate}
\usepackage{aliascnt}
\usepackage{booktabs}
\usepackage[bookmarks=true,
            pdftitle={Cartan matrices and Brauer's k(B)-Conjecture V},
            pdfauthor={Benjamin Sambale},
            pdfkeywords={number of characters, Usami-Puig method},
            pdfstartview={FitH}]{hyperref}

\newtheorem{Thm}{Theorem} 
\newaliascnt{Lem}{Thm}
\newtheorem{Lem}[Lem]{Lemma}
\aliascntresetthe{Lem}
\newaliascnt{Prop}{Thm}
\newtheorem{Prop}[Prop]{Proposition}
\aliascntresetthe{Prop}

\numberwithin{equation}{section}

\allowdisplaybreaks[1]

\renewcommand{\phi}{\varphi}
\newcommand{\C}{\mathrm{C}}
\newcommand{\N}{\mathrm{N}}
\newcommand{\Z}{\mathrm{Z}}
\newcommand{\ZZ}{\mathbb{Z}}
\newcommand{\CC}{\mathbb{C}}
\newcommand{\QQ}{\mathbb{Q}}
\newcommand{\RR}{\mathbb{R}}

\newcommand{\cohom}{\operatorname{H}}
\newcommand{\Aut}{\operatorname{Aut}}

\newcommand{\GL}{\operatorname{GL}}
\newcommand{\SL}{\operatorname{SL}}

\newcommand{\PSL}{\operatorname{PSL}}

\newcommand{\Sz}{\operatorname{Sz}}
\newcommand{\Irr}{\mathrm{Irr}}
\newcommand{\IBr}{\mathrm{IBr}}

\newcommand{\Syl}{\operatorname{Syl}}

\newcommand{\AGL}{\operatorname{AGL}}
\newcommand{\ASL}{\operatorname{ASL}}
\newcommand{\PGL}{\operatorname{PGL}}
\newcommand{\AGammaL}{\operatorname{A\Gamma L}}
\newcommand{\GammaL}{\operatorname{\Gamma L}}
\newcommand{\Qd}{\mathrm{Qd}}
\newcommand{\tr}{\mathrm{tr}}
\newcommand{\PIM}{\operatorname{PIM}}

\mathchardef\ordinarycolon\mathcode`\:  
 \mathcode`\:=\string"8000
 \begingroup \catcode`\:=\active
   \gdef:{\mathrel{\mathop\ordinarycolon}}
 \endgroup

\title{Cartan matrices and\\ Brauer's $k(B)$-Conjecture V}
\author{Cesare Giulio Ardito\footnote{School of Mathematics, University of Manchester, Manchester, M13 9PL, United Kingdom, \href{mailto:cesareg.ardito@gmail.com}{cesareg.ardito@gmail.com}} \ and Benjamin Sambale\footnote{Institut für Algebra, Zahlentheorie und Diskrete Mathematik, Leibniz Universität Hannover, Welfengarten 1, 30167 Hannover, Germany,
\href{mailto:sambale@math.uni-hannover.de}{sambale@math.uni-hannover.de}}}
\date{\today}

\begin{document}
\frenchspacing
\maketitle
\begin{abstract}\noindent
We prove Brauer's $k(B)$-Conjecture for the $3$-blocks with abelian defect groups of rank at most $5$ and for all $3$-blocks of defect at most $4$. For this purpose we develop a computer algorithm to construct isotypies based on a method of Usami and Puig. This leads further to some previously unknown perfect isometries for the $5$-blocks of defect $2$. We also investigate basic sets which are compatible under the action of the inertial group.
\end{abstract}


\section{Introduction}

This work continues a series of articles the last one being \cite{SambaleC4}. Before stating the main theorems we briefly explain the strategy behind all papers in this series.
 
Let $B$ be a block of a finite group $G$ with respect to an algebraically closed field $F$ of characteristic $p>0$. Let $D$ be a defect group $B$, and let $k(B):=|\Irr(B)|$ and $l(B):=|\IBr(B)|$. To prove Brauer's $k(B)$-Conjecture, that $k(B)\le|D|$, we investigate Brauer correspondents of $B$ in local subgroups. More precisely, let $z\in\Z(D)$ and let $b_z$ be a Brauer correspondent of $B$ in $\C_G(z)$. If we can determine the Cartan matrix $C_z$ of $b_z$ up to basic sets (i.\,e. up to transformations $C_z\to SC_zS^\mathrm{t}$ where $S\in\GL(l(b_z),\ZZ)$), then Brauer's Conjecture usually follows from \cite[Theorem~4.2]{habil} or from the much stronger result \cite[Theorem~A]{SambaleBound}. Now $b_z$ dominates a unique block $\overline{b_z}$ of $\C_G(z)/\langle z\rangle$ with Cartan matrix $\overline{C_z}=\frac{1}{|\langle z\rangle|}C_z$. Hence, it suffices to consider $\overline{b_z}$. By \cite[Lemma~3]{SambaleC4}, $\overline{b_z}$ has defect group $\overline{D}:=D/\langle z\rangle$ and the fusion system of $\overline{b_z}$ is uniquely determined by the fusion system of $B$. This means that we have full information on $\overline{b_z}$ on the local level. The inertial quotient of $B$ is denoted by $I(B)$ in the following. 

In the present paper we deal mostly with situations where $\overline{D}$ is abelian. Then the fusion system of $\overline{b_z}$ is essentially determined by the inertial quotient $I(\overline{b_z})=I(b_z)\cong\C_{I(B)}(z)$ and by the action of $I(\overline{b_z})$ on $\overline{D}$. In the next section we will revisit a method developed by Usami and Puig to construct perfect isometries between $\overline{b_z}$ and its Brauer first main theorem correspondent in certain situations. Since perfect isometries preserve Cartan matrices (up to basic sets), it suffices to determine the Cartan matrix of a block $\beta_z$ with normal defect group $\overline{D}$ and $I(\beta_z)=I(b_z)$. By a theorem of Külshammer~\cite{Kuelshammer}, we may even assume that $\beta_z$ is a twisted group algebra of $L:=\overline{D}\rtimes I(b_z)$. Finally, we can regard $\beta_z$ as a faithful block of a certain central extension $\widehat{L}$ of $L$ by a cyclic $p'$-group. It is then straight-forward to compute the desired Cartan matrix.

In the third section we apply a novel computer implementation of the Usami--Puig method to construct many new isotypies for $5$-blocks of defect $2$. This verifies Broué's Abelian Defect Group Conjecture~\cite{Broue} on the level of characters in those cases.
When this approach fails, it is often still possible to determine a short list of all potential Cartan matrices of $b_z$. To do so, we improve the Cartan method introduced in \cite[Section~4.2]{habil}. As a new ingredient we investigate in \autoref{basicsets} the existence of basic sets which are compatible with the action of the inertial group.
Eventually, we combine both methods to verify Brauer's $k(B)$-Conjecture for all $3$-blocks with abelian defect groups of rank at most $5$. This extends the corresponding results from \cite[Proposition 21]{SambaleC4} and \cite[Corollary~3]{SambaleRank3} for $p$-blocks with abelian defect groups of rank at most $3$ (respectively $7$ if $p=2$).
Afterwards we turn the focus to non-abelian defect groups. In the last section we prove Brauer's Conjecture under the hypothesis that $\overline{D}$ is metacyclic. This result relies on a recent paper by Tasaka--Watanabe~\cite{TasakaWatanabe}.
Finally, a careful analysis shows that Brauer's Conjecture holds for all defect groups of order $3^4$. We remark that Brauer's Conjecture for $p$-blocks of defect $3$ has been verified previously in \cite[Theorem~B]{SambaleRank3} for arbitrary $p$.

Although our methods are of elementary nature they crucially rely on one direction of Brauer's Height Zero Conjecture proven by Kessar--Malle~\cite{KessarMalle} via the classification of finite simple groups.

\section{A method of Usami and Puig}

In addition to the notation already introduced we follow mostly \cite{habil}. To distinguish cyclic groups from Cartan matrices and centralizers we denote them by $Z_n$. The symmetric and alternating groups of degree $n$ as well as the dihedral, semidihedral and quaternion groups of order $n$ are denoted by $S_n$, $A_n$, $D_n$, $SD_n$ and $Q_n$ respectively. Moreover, we make use of the Mathieu group $M_9\cong Z_3^2\rtimes Q_8$ (a sharply $2$-transitive group of degree $9$). A central product of groups $G$ and $H$ is denoted by $G*H$. 
The Kronecker $\delta_{ij}$ (being $1$ if $i=j$ and $0$ otherwise) is often used to write matrices in a concise form. Finally, a \emph{basic set} of a block $B$ is a $\ZZ$-basis of the Grothendieck group $\ZZ\IBr(B)$ of generalized Brauer characters of $B$. 

In this section we assume that $B$ is a block with \emph{abelian} defect group $D$. Let $b$ be a Brauer correspondent of $B$ in $\C_G(D)$.
We regard the inertial quotient $E:=I(B)=\N_G(D,b)/\C_G(D)$ as a subgroup of $\Aut(D)$. Let $L:=D\rtimes E$.
It is well-known that $b$ is nilpotent and $\IBr(b)=\{\phi\}$. Now $\phi$ gives rise to a projective representation $\Gamma$ of the inertial group $\N_G(D,b)$ (see \cite[Theorem~8.14]{Navarro}). Moreover, $\Gamma$ is associated to a $2$-cocycle $\gamma$ of $E$ with values in $F^\times$ (see \cite[Theorem~8.15]{Navarro}). 
Külshammer's result mentioned above states that $b^{\N_G(D)}$ is Morita equivalent to the twisted group algebra $F_\gamma L$ (note that $b^{\N_G(D)}$ and $b^{\N_G(D,b)}$ are Morita equivalent by the Fong--Reynolds theorem). 

In several papers (starting perhaps with \cite{Usami23I}), Usami and Puig developed an inductive method to establish an isotypy between $B$ and $F_\gamma L$. This is a family of compatible perfect isometries between $b^{\C_G(Q)}$ and $F_\gamma\C_L(Q)$ for every $Q\le D$. In particular, for $Q=1$ we obtain a perfect isometry between $B$ and $F_\gamma L$. In the following we introduce the necessary notation. For $Q\le D$ and $H\le\N_G(Q)$ we write $\overline{H}:=HQ/Q$. Let $b_Q:=b^{\C_G(Q)}$ and let $\overline{b_Q}$ be the unique block of $\overline{\C_G(Q)}$ dominated by $b_Q$. For any (twisted) group algebra or block $A$ let $\ZZ\Irr(A)$ be the Grothendieck group of generalized characters of $A$. Let $\ZZ\Irr^0(A)$ be the subgroup of $\ZZ\Irr(A)$ consisting of the generalized characters which vanish on the $p$-regular elements of the corresponding group. For class functions $\chi$ and $\psi$ on $G$ we use the usual scalar product
\begin{equation}\label{scalar}
(\chi,\psi):=\frac{1}{|G|}\sum_{g\in G}\chi(g)\overline{\psi(g)}.
\end{equation}

We note that $\N_E(Q)$ acts naturally on $\ZZ\Irr^0(F_\gamma \overline{\C_L(Q)})$ and $\ZZ\Irr(F_\gamma \overline{\C_L(Q)})$ as well as on $\ZZ\Irr^0(\overline{b_Q})$ and $\ZZ\Irr(\overline{b_Q})$. A map $f$ between any two of these sets is called $\N_E(Q)$-\emph{equivariant} if $f(\chi)^e=f(\chi^e)$ for every $e\in\N_E(Q)$ and every generalized character $\chi$ in the respective set.
By \cite[Proposition~3.11 and Section~4.3]{UsamiZ2Z2} (compare with \cite[Section~3.4]{UsamiD6}), it suffices to show that a given $\N_E(Q)$-equivariant bijective isometry (with respect to \eqref{scalar})
\[\Delta_0:\ZZ\Irr^0(F_\gamma \overline{\C_L(Q)})\to\ZZ\Irr^0(\overline{b_Q})\]
extends to an $\N_E(Q)$-equivariant isometry
\begin{equation}\label{UP}
\Delta:\ZZ\Irr(F_\gamma\overline{\C_L(Q)})\to\ZZ\Irr(\overline{b_Q})
\end{equation}
($\Delta$ will automatically be surjective). To prove this, we may replace $(G,B,E)$ by 
\[\bigl(\N_G(Q,b_Q),b_Q^{\N_G(Q,b_Q)},\N_E(Q)\bigr)\] 
in order to argue by induction on $|E|$.
For example, if $|\N_E(Q)|\le 4$ or $\N_E(Q)\cong S_3$, then the claim holds by the main theorems of \cite{UsamiZ2Z2,UsamiZ4,Usami23I,UsamiD6}. Furthermore, the claim holds for $Q=D$ as shown in \cite[3.4.2]{UsamiZ2Z2}. 

In the following we assume that $Q<D$ is given. It is straight-forward to determine from the character table a $\ZZ$-basis $\rho_1,\ldots,\rho_m$ of $\ZZ\Irr^0(F_\gamma \overline{\C_L(Q)})$. Let $\chi_1,\ldots,\chi_k\in\Irr(F_\gamma\overline{\C_L(Q)})$ and $A=(a_{ij})\in\ZZ^{k\times m}$ such that
$\rho_i=\sum_{j=1}^ka_{ji}\chi_j$ for $i=1,\ldots,m$. 
Let $\widehat\rho_i:=\Delta_0(\rho_i)$ for $i=1,\ldots,m$. 
Since $\Delta_0$ is an isometry, we have 
\[C_*:=A^\mathrm{t}A=(\rho_i,\rho_j)_{1\le i,j\le m}=(\widehat\rho_i,\widehat\rho_j)_{1\le i,j\le m}\]
where $A^\mathrm{t}$ denotes the transpose of $A$. 
The matrix equation $Q_*^\mathrm{t}Q_*=C_*$ can be solved with an algorithm of Plesken~\cite{Plesken} which is implemented in GAP~\cite{GAP48} (command \texttt{OrthogonalEmbeddings}). We will see that in many situations there is only one solution up to permutations and signs of the rows of $Q_*$. This implies that there exist $\widehat\chi_1,\ldots,\widehat\chi_k\in\pm\Irr(\overline{b_Q})$ such that $\widehat\rho_i=\sum_{j=1}^ka_{ji}\widehat\chi_j$ for $i=1,\ldots,m$.
Then the isometry $\Delta$ defined by $\Delta(\chi_i):=\widehat\chi_i$ for $i=1,\ldots,k$ clearly extends $\Delta_0$. 
If $\N_E(Q)=\C_E(Q)$, then $\Delta$ is always $\N_E(Q)$-equivariant. This holds in particular if $Q=1$. 
In several other cases we can show that the rows of $Q_*$ are pairwise linearly independent (i.\,e. $r\ne \pm s$ for distinct rows $r,s$). It follows that $\Delta$ is in fact the only extension of $\Delta_0$ (note that $-\Delta$ does not extend $\Delta_0$ since we are assuming $Q<D$). Now for every $e\in\N_E(Q)$ the map $\widetilde{\Delta}:\ZZ\Irr(F_\gamma\overline{\C_L(Q)})\to\ZZ\Irr(\overline{b_Q})$, $\chi\mapsto{^{e^{-1}}(\Delta({^e\chi}))}$ is also an isometry extending $\Delta_0$. Therefore, $\widetilde{\Delta}=\Delta$ and $\Delta$ is $\N_E(Q)$-equivariant.

Since the generalized characters $\widehat\rho_i$ vanish on the $p$-regular elements, these characters are orthogonal to the projective indecomposable characters of $\overline{b_Q}$. In other words, the columns of $Q_*$ are orthogonal to the columns of the decomposition matrix of $\overline{b_Q}$.
In order to reduce the number of possible solutions of the equation $Q_*^\mathrm{t}Q_*=C_*$, we prove the following result. 

\begin{Lem}\label{lemM}
Let $B$ be a $p$-block of a finite group $G$ with abelian defect group $D\ne 1$ and decomposition matrix $Q_1\in\ZZ^{k\times l}$. Let $Q_*\in\ZZ^{k\times(k-l)}$ be of rank $k-l$ such that $Q_1^\mathrm{t}Q_*=0$. Let $C_*:=Q_*^\mathrm{t}Q_*$. Then for every row $r$ of $Q_*$ we have $|D|rC_*^{-1}r^\mathrm{t}\in\{1,\ldots,|D|\}\setminus p\ZZ$.
\end{Lem}
\begin{proof}
Let $C:=Q_1^\mathrm{t}Q_1$ be the Cartan matrix of $B$. Since $Q_1$ and $Q_*$ have full rank, the matrix $R:=(Q_1,Q_*)\in\ZZ^{k\times k}$ is invertible. We compute
\[1_k=R(R^\mathrm{t}R)^{-1}R^\mathrm{t}=(Q_1,Q_*)\begin{pmatrix}
C^{-1}&0\\
0&C_*^{-1}
\end{pmatrix}\begin{pmatrix}
Q_1^\mathrm{t}\\
Q_*^{\mathrm{t}}
\end{pmatrix}=Q_1C^{-1}Q_1^\mathrm{t}+Q_*C_*^{-1}Q_*^{\mathrm{t}}.\]
It is well-known that $|D|C^{-1}$ and $|D|Q_1C^{-1}Q_1^\mathrm{t}$ are integer matrices (see \cite[Theorem~3.26]{Navarro}). Hence, $|D|Q_*C_*^{-1}Q_*^{\mathrm{t}}$ is also an integer matrix and the (non-negative) diagonal entries are bounded by $|D|$. By Kessar--Malle~\cite{KessarMalle}, all irreducible characters in $B$ have height $0$. By a result of Brauer (see \cite[Proposition~1.36]{habil}), it follows that the diagonal entries of $|D|Q_1C^{-1}Q_1^\mathrm{t}$ are not divisible by $p$. Hence, the same must hold for the diagonal entries of $|D|Q_*C_*^{-1}Q_*^{\mathrm{t}}$. The claim follows.
\end{proof}

Sometimes we know a priori that $l(F_\gamma\overline{\C_L(Q)})=l(\overline{b_Q})$ (for instance, if $\overline{D}$ is cyclic or $|\C_E(Q)|\le 4$ by the results of Usami--Puig cited above). Since $\Delta_0$ is an isomorphism, we also have \[k(F_\gamma\overline{\C_L(Q)})-l(F_\gamma\overline{\C_L(Q)})=k(\overline{b_Q})-l(\overline{b_Q}).\] 
Hence, we can restrict Plesken's algorithm to those $Q_*$ which have exactly $k(F_\gamma\overline{\C_L(Q)})$ rows.
In this favorable situation the Grothendieck groups $\ZZ\PIM(F_\gamma\overline{\C_L(Q)})$ and $\ZZ\PIM(\overline{b_Q})$ spanned by the projective indecomposable characters have the same rank.
Since $\ZZ\PIM(.)$ is the orthogonal complement of $\ZZ\Irr^0(.)$ in $\ZZ\Irr(.)$, it suffices to construct an $\N_E(Q)$-equivariant isometry $\ZZ\PIM(F_\gamma\overline{\C_L(Q)})\to \ZZ\PIM(\overline{b_Q})$ which can then be combined with $\Delta_0$ to obtain $\Delta$. This alternative strategy is pursued in \autoref{defect2} below.

The entire procedure can be executed by GAP without human intervention. In fact, hand calculations of this kind become very tedious and are prone to errors. We summarize our algorithm under the assumption that $L:=D\rtimes E$ is given.

\begin{enumerate}[(1)]
\item Determine the Schur multiplier $H:=\cohom^2(E,\CC^\times)$.
\item For every cyclic subgroup $Z\le H$ do the following
\begin{enumerate}[(a)]
\item Construct a stem extension $\widehat{L}$ of $L$ such that $\widehat{L}/Z\cong L$.
\item Determine a set $\mathcal{Q}$ of representatives for the $L$-conjugacy classes of subgroups $Q< D$ such that 
$|\N_E(Q)|>4$ and $\N_E(Q)\not\cong S_3$.
\item For every $Q\in\mathcal{Q}$ and every faithful block $\beta$ of $Y:=\C_{\widehat{L}}(Q)/Q$ do the following:
\begin{enumerate}[(i)]
\item Determine the matrix $A:=(\chi_i(y_j))_{i,j}$ where $\Irr(\beta)=\{\chi_1,\ldots,\chi_k\}$ and $y_1,\ldots,y_l$ are representatives for the conjugacy classes of $p'$-elements of $Y$.
\item Compute a $\ZZ$-basis $u_1,\ldots,u_{k-l}$ of the orthogonal space $\{v:\in\ZZ^k:vA=0\}$ (using the Smith normal form for instance). 
\item Compute $C_*=(u_iu_j^\mathrm{t})_{i,j=1}^{k-l}$.
\item Determine the (finite) set $\mathcal{R}$ of rows $r\in\ZZ^{k-l}$ such that \[|D/Q|rC_*^{-1}r^{\mathrm{t}}\in\{1,\ldots,|D/Q|\}\setminus p\ZZ.\] 
\item Apply Plesken's algorithm to solve $C_*=Q_*^\mathrm{t}Q_*$ such that every row of $Q_*$ belongs to $\mathcal{R}$.
\item If there is a unique solution $Q_*$ up to permutations and signs of rows, then $\Delta_0$ extends to some isometry $\Delta$.
\item If $\N_E(Q)=\C_E(Q)$, then $\Delta$ is $\N_E(Q)$-equivariant.
\item If the rows of $Q_*$ are pairwise linearly independent, then $\Delta$ is $\N_E(Q)$-equivariant.
\item Deal with the exceptions.
\end{enumerate}
\end{enumerate}
\end{enumerate}

In view of the fact that Kessar--Malle's result was not available to Usami and Puig, it is not surprising that our approach goes beyond their results. For instance, the isotypies for $D\cong Z_2\times Z_2\times Z_2$ constructed by Kessar--Koshitani--Linckelmann~\cite{KKL} (also relying on the classification of finite simple groups) can now be obtained by pressing a button. 
Our algorithm in combination with \cite[Proposition~13.4]{habil} also applies to $L\cong A_4\times A_4$ ($p=2$) and therefore simplifies and improves the main result of \cite{LS}. In fact, an extension of this case was recently settled by the first author in \cite[Proposition~3.3]{Ardito}. We should however also mention that the computational complexity of Plesken's algorithm grows rapidly with the size of the involved matrices.

\section{Blocks of defect 2}

One approach to classify blocks $B$ with a given defect group $D$ is to distribute them into families such that each family corresponds to a Morita equivalence class of the Brauer correspondent $B_D$ of $B$ in $\N_G(D)$ (there are only finitely many choices for these classes). If $D$ is cyclic, this has been accomplished by using the Brauer tree. Also the blocks with Klein four defect group $D\cong Z_2\times Z_2$ are completely classified.
The group $D\cong Z_3\times Z_3$ has first been investigated by Kiyota~\cite{Kiyota} in 1984 and is still not fully understood today. We will recap the details and further study $D\cong Z_5\times Z_5$ in this section.

We begin by showing that only the subgroup $Q=1$ in Usami--Puig's methods needs to be considered. This fact is related to the existence of a stable equivalence of Morita type stated in \cite[Section~6.2]{RouquierStable}.

\begin{Prop}\label{defect2}
Let $B$ be a block of a finite group $G$ with defect group $D\cong C_p\times C_p$ and cocycle $\gamma$ as in the previous section. Let $L:=D\rtimes I(B)$. Suppose that every $I(B)$-equivariant isometry $\Delta_0:\ZZ\Irr^0(F_\gamma L)\to \ZZ\Irr^0(B)$ extends to an $I(B)$-equivariant isometry $\Delta:\ZZ\Irr(F_\gamma L)\to \ZZ\Irr(B)$. Then $B$ is isotypic to its Brauer correspondent in $\N_G(D)$.
\end{Prop}
\begin{proof}
Let $Q\le D$ be of order $p$. We need to show the existence of $\Delta$ with respect to $Q$ as in \eqref{UP}. To this end we may assume that $G=\N_G(Q,b_Q)$ and $E:=I(B)$ normalizes $Q$. 
Let $\widehat{L}$ be a suitable stem extension such that $F_\gamma\overline{\C_L(Q)}$ is isomorphic to a block $\beta_Q$ of $\C_{\widehat{L}}(Q)/Q$. Observe that $\beta_Q$ and $\overline{b_Q}$ have defect $1$ and inertial quotient $\C_E(Q)$. 
By Brauer's theory of blocks of defect $1$, we have $l:=l(\beta_Q)=|\C_E(Q)|=l(\overline{b_Q})$. 
Since $G/\C_G(Q)\cong E/\C_E(Q)$ is cyclic, \cite[Lemma~3.3]{Sambalerefine} (or \autoref{basicsets} below) implies the existence of a basic set $\Phi$ of $\overline{b_Q}$ such that $\IBr(\overline{b_Q})$ and $\Phi$ are isomorphic $E$-sets and the Cartan matrix of $\overline{b_Q}$ with respect to $\Phi$ is $C:=(m+\delta_{ij})_{i,j=1}^l$ where $m:=(p-1)/l$. This is also the Cartan matrix of $\beta_Q$ (with respect to $\IBr(\beta_Q)$). Let $Q=(d_{\chi\phi})$ be the decomposition matrix of $\overline{b_Q}$ with respect to $\Phi$. For $\phi\in\Phi$ we define the projective character $\widehat{\phi}:=\sum_{\chi\in\Irr(\overline{b_Q})}d_{\chi\phi}\chi$. By the shape of the matrix $C$, every bijection between $\PIM(\beta_Q)$ and $\{\widehat{\phi}:\phi\in\Phi\}$ induces an isometry  $\ZZ\PIM(\beta_Q)\to\ZZ\PIM(\overline{b_Q})$. Since $\ZZ\PIM(\beta_Q)$ is the orthogonal complement of $\ZZ\Irr^0(\beta_Q)$, we can extend $\Delta_0$ in this way to an isometry $\Delta:\ZZ\Irr(\beta_Q)\to\ZZ\Irr(\overline{b_Q})$. In order to make $\Delta$ $E$-equivariant, it suffices to show that $\IBr(\beta_Q)$ and $\IBr(\overline{b_Q})$ are isomorphic $E$-sets.

By \cite[Proposition~3.14]{UsamiZ2Z2} there exists a bijection between the set of blocks of $\widehat{L}/Q$ covering $\beta_Q$ and the set of blocks of $\overline{G}$ covering $\overline{b_Q}$. Moreover, this bijection preserves defect groups and inertial quotients. Since the blocks in both sets (still) have defect $1$, the number of irreducible Brauer characters is uniquely determined by the respective inertial indices. Consequently, the number of Brauer characters of $\widehat{L}/Q$ lying over $\beta_Q$ coincides with the number of Brauer characters of $\overline{G}$ lying over $\overline{b_Q}$. 
We claim that this number uniquely determines the action of $E$ on $\IBr(\beta_Q)$ and on $\IBr(\overline{b_Q})$.
Since $\widehat{L}/\C_{\widehat{L}}(Q)\cong G/\C_G(Q)\cong E/\C_E(Q)$ is cyclic, every $\phi\in\IBr(\beta_Q)\cup\IBr(\overline{b_Q})$ extends to its inertial group (see \cite[Theorem~8.12]{Navarro}). Moreover by \cite[Proposition~3.2]{Sambalerefine}, $E$ acts $\frac{1}{2}$-transitively on $\IBr(\beta_Q)$ and on $\IBr(\overline{b_Q})$. This means that all orbits on $\IBr(\beta_Q)$ have a common length, say $d_L$, and similarly all orbits on $\IBr(\overline{b_Q})$ have length, say $d_G$. By Clifford theory, there are exactly $l(\beta_Q)|E/\C_E(Q)|/d_L^2$ irreducible Brauer characters in $\widehat{L}/Q$ lying over $\beta_Q$. Similarly, there are $l(\overline{b_Q})|E/\C_E(Q)|/d_G^2$ Brauer characters in $\overline{G}$ lying over $\overline{b_Q}$. Since $l(\beta_Q)=l(\overline{b_Q})$ we conclude that $d_L=d_G$. Thus, $\IBr(\beta)$ and $\IBr(\overline{b_Q})$ are isomorphic $E$-sets. 
\end{proof}

The following result on the case $p=3$ is mostly well-known, but hard to find explicitly in the literature. The column \emph{group} in \autoref{local3} refers to the small group library in GAP. If this group has an easy structure, then it is described in the \emph{comments} column. If the comment is \emph{non-principal}, then the group is a double cover of the preceding group in the list and the block is the unique non-principal block. In the remaining cases, the group has only one block (the principal block). Finally the column \emph{isotypy} indicates if an isotypy between $B$ and $B_D$ is known to exist. 

\begin{Thm}\label{local3}
Let $B$ be a block of a finite group $G$ with defect group $D\cong Z_3\times Z_3$. Then the Brauer correspondent $B_D$ of $B$ in $\N_G(D)$ is Morita equivalent to exactly one of the following blocks:
\begin{center}
\begin{tabular}{*7{c}}
\toprule 
no.&$I(B)$&group&$k(B_D)$&$l(B_D)$&isotypy&comments\\\midrule
$1$&$1$&$9:2$&$9$&$1$&\checkmark&$D$, nilpotent\\
$2$&$Z_2$&$18:3$&$9$&$2$&\checkmark&$S_3\times Z_3$\\
$3$&$Z_2$&$18:4$&$6$&$2$&\checkmark&Frobenius group\\
$4$&$Z_2^2$&$36:10$&$9$&$4$&\checkmark&$S_3^2$\\
$5$&$Z_2^2$&$72:23$&$6$&$1$&\checkmark&non-principal\\
$6$&$Z_4$&$36:9$&$6$&$4$&\checkmark&Frobenius group\\
$7$&$Z_8$&$72:39$&$9$&$8$&&$\AGL(1,9)$\\
$8$&$Q_8$&$72:41$&$6$&$5$&&$M_9$\\
$9$&$D_8$&$72:40$&$9$&$5$&\checkmark&$S_3\wr Z_2$\\
$10$&$D_8$&$144:117$&$6$&$2$&\checkmark&non-principal\\
$11$&$SD_{16}$&$144:182$&$9$&$7$&\checkmark&$\AGammaL(1,9)$\\\bottomrule
\end{tabular} 
\end{center}
\end{Thm}
\begin{proof}
Since $\Aut(D)\cong\GL(2,3)$ has order $16\cdot 3$, $E:=I(B)$ is a subgroup of the semidihedral group $SD_{16}\cong\GammaL(1,9)\in\Syl_2(\GL(2,3))$.
As explained above, $B_D$ is Morita equivalent to a twisted group algebra $F_\gamma[D\rtimes E]$. If $E\notin\{Z_2^2,D_8\}$, then $E$ has trivial Schur multiplier and $\gamma=1$. In this case we list $L:=D\rtimes E$ in the group column and compute $k(B_D)=k(L)$ and $l(B_D)=k(E)$. If, on the other hand, $E\in\{Z_2^2,D_8\}$, then the Schur multiplier of $E$ is $Z_2$. Thus, there is exactly one non-trivial twisted group algebra in each case. Here we compute $l(B_D)=k(\widehat{E})-k(E)$ where $\widehat{E}$ is a double cover of $E$. The isotypies can be obtained with our algorithm from the last section (cf. \cite[Proposition~6.3]{SambaleIso}). 

It remains to show that each two different cases in our list are not Morita equivalent. This is clear from the computed invariants except for the cases $3$ and $10$. 
By \cite[Corollary~3.5]{BroueEqui}, a Morita equivalence preserves the isomorphism type of the stable center $\overline{\Z}(B_D)$. In case $3$, this algebra is symmetric by \cite[Theorem~1.1]{FrobeniusInertial}. Now we use \cite[Theorem~3.1]{FrobeniusInertial} in order to show that the stable center is not symmetric in case $10$. 
The group $E\cong D_8$ has two orbits on $D\setminus\{1\}$. Hence, there exists two non-trivial $B_D$-subsections $(u,\beta_u)$ and $(v,\beta_v)$ up to conjugation. Brauer's formula (see \cite[Theorem 1.35]{habil}) gives $4=k(B_D)-l(B_D)=l(\beta_u)+l(\beta_v)$. Hence, we may assume that $l(\beta_u)>1$ and the claim follows from \cite[Theorem~3.1]{FrobeniusInertial}.
\end{proof}

The reason why the Usami--Puig method fails for $I(B)\in\{Z_8,Q_8\}$ is because in both cases one gets $C_*=(9)$. For $I(B)\cong Z_8$, $C_*$ factors into $Q_*=(\pm1,\ldots,\pm1)^\mathrm{t}$ and for $I(B)\cong Q_8$ we have $Q_*=(\pm2,\pm1,\ldots,\pm1)^\mathrm{t}$. Once we know that $k(B)=k(B_D)$, then $B$ and $B_D$ must be isotypic. Even worse, it seems to be open whether $Q_*=(\pm2,\pm2,\pm1)^\mathrm{t}$ can actually occur. Equivalently, does there exist a block $B$ with defect group $D\cong Z_3\times Z_3$ and $k(B)=3$?

We remark that $B$ is usually not Morita equivalent to $B_D$. For principal blocks the possible Morita equivalence classes for $B$ were obtain by Koshitani~\cite{KoshitaniDonovanAbel}. The column \emph{type of $B_D$} in the following table refers to the numbering in \autoref{local3}.

\begin{Thm}[Koshitani]\label{koshitani}
Let $B$ be the principal block of a finite group $G$ with defect group $D\cong Z_3\times Z_3$ and $B_D$ be the principal $3$-block of $\N_G(D)$. Then $B$ is (splendid) Morita equivalent either to one of the nine principal cases in \autoref{local3} or to exactly one of the following principal blocks:
\begin{center}
\begin{tabular}{*3{c}}
\toprule 
no.&type of $B_D$&group\\\midrule
$12$&$6$&$A_6$\\
$13$&$6$&$A_7$\\
$14$&$7$&$\PGL(2,9)$\\
$15$&$8$&$M_{10}$\\
$16$&$8$&$\PSL(3,4)$\\
$17$&$9$&$S_6$\\
$18$&$9$&$S_7$\\
$19$&$9$&$A_8$\\
$20$&$11$&$M_{11}$\\
$21$&$11$&$HS$\\
$22$&$11$&$M_{23}$\\
$23$&$11$&$\PSL(3,4).2^2$\\
$24$&$11$&$\Aut(S_6)$\\\bottomrule
\end{tabular} 
\end{center}
\end{Thm}
\begin{proof}
By \cite{KoshitaniDonovanAbel,KoshitaniCorr}, $B$ is splendidly Morita equivalent to one of the given blocks. 
Using the GAP command \texttt{TransformingPermutations}, one can check that each two of those blocks have essentially different decomposition matrices. Hence, they cannot be Morita equivalent (splendid or not).
\end{proof}

According to Scopes~\cite[Example~2 on p. 455]{Scopes}, every block $B$ of a symmetric group with defect group $D\cong Z_3\times Z_3$ is Morita equivalent to the principal block of $S_6$, $S_7$, $S_8$, to the “second” block of $S_8$, or to the third block of $S_{11}$. The first and second block of $S_8$ are both isomorphic to the principal block of $A_8$ via restriction of characters (see \cite[Théor{\`e}me~0.1]{Broueiso}). The block of $S_{11}$ is a RoCK block and must be Morita equivalent to its Brauer correspondent. Hence, $B$ always belongs to one of $24$ blocks in the above theorems. 
 
Nevertheless, we found twelve further Morita equivalence classes among the non-principal blocks while checking the character library in GAP. For instance, a non-principal block of the double cover $2.A_6$. Recall that according to Donovan's Conjecture the total number of Morita equivalence classes of blocks with defect group $D$ should be finite.

Now we turn to $p=5$. In the table below the examples are always faithful blocks of the given group. The Morita equivalence class of such a block is indeed uniquely determined as we will see in the proof. In order to distinguish Morita equivalence classes, we also list the multiplicity $c(B_D)$ of $1$ as an elementary divisor of the Cartan matrix of $B_D$ and the Loewy length $LL(ZB_D)$ of the center of $B_D$ (considered as an $F$-algebra). 

\begin{Thm}\label{local5}
Let $B$ be a block of a finite group $G$ with defect group $D\cong Z_5\times Z_5$. Then the Brauer correspondent $B_D$ of $B$ in $\N_G(D)$ is Morita equivalent to exactly one of the following blocks:
\newcounter{case}
\begin{small}
\begin{center}
\begin{tabular}{*9{c}}
\toprule  
no.&$I(B)$&group&$k(B_D)$&$l(B_D)$&$c(B_D)$&$LL(ZB_D)$&isotypy&comments\\\midrule
\stepcounter{case}$\thecase$&$1$&$25:2$&$25$&$1$&$0$&$9$&\checkmark&$D$, nilpotent\\
\stepcounter{case}$\thecase$&$Z_2$&$50:3$&$20$&$2$&$0$&$7$&\checkmark&$D_{10}\times Z_5$\\ %
\stepcounter{case}$\thecase$&$Z_2$&$50:4$&$14$&$2$&$1$&$5$&\checkmark&Frobenius group\\ 
\stepcounter{case}$\thecase$&$Z_3$&$75:2$&$11$&$3$&$2$&$5$&\checkmark&Frobenius group\\ 
\stepcounter{case}$\thecase$&$Z_4$&$100:9$&$25$&$4$&$0$&$6$&\checkmark&$(Z_5\rtimes Z_4)\times Z_5$\\ 
\stepcounter{case}$\thecase$&$Z_4$&$100:10$&$13$&$4$&$2$&$4$&\checkmark&\\
\stepcounter{case}$\thecase$&$Z_4$&$100:11$&$10$&$4$&$3$&$3$&\checkmark&Frobenius group\\
\stepcounter{case}$\thecase$&$Z_4$&$100:12$&$10$&$4$&$3$&$5$&\checkmark&Frobenius group\\ 
\stepcounter{case}$\thecase$&$Z_2^2$&$100:13$&$16$&$4$&$1$&$5$&\checkmark&$D_{10}^2$\\ 
\stepcounter{case}$\thecase$&$Z_2^2$&$200:24$&$13$&$1$&$0$&$5$&\checkmark&non-principal\\
\stepcounter{case}$\thecase$&$C_6$&$150:6$&$10$&$6$&$5$&$5$&&Frobenius group\\ 
\stepcounter{case}$\thecase$&$S_3$&$150:5$&$13$&$3$&$1$&$5$&\checkmark&\\
\stepcounter{case}$\thecase$&$Z_8$&$200:40$&$11$&$8$&$7$&$3$&&Frobenius group\\
\stepcounter{case}$\thecase$&$Z_4\times Z_2$&$200:41$&$20$&$8$&$3$&$4$&\checkmark&$(Z_5\rtimes Z_4)\times D_{10}$\\
\stepcounter{case}$\thecase$&$Z_4\times Z_2$&$400:118$&$14$&$2$&$0$&$4$&\checkmark&non-principal\\ 
\stepcounter{case}$\thecase$&$Z_4\times Z_2$&$200:42$&$14$&$8$&$5$&$3$&&\\
\stepcounter{case}$\thecase$&$Z_4\times Z_2$&$400:125$&$8$&$2$&$1$&$3$&\checkmark&non-principal\\
\stepcounter{case}$\thecase$&$Q_8$&$200:44$&$8$&$5$&$4$&$3$&&Frobenius group\\
\stepcounter{case}$\thecase$&$D_8$&$200:43$&$14$&$5$&$2$&$5$&\checkmark&$D_{10}\wr Z_2$\\ 
\stepcounter{case}$\thecase$&$D_8$&$400:131$&$11$&$2$&$1$&$5$&\checkmark&non-principal\\ 
\stepcounter{case}$\thecase$&$Z_{12}$&$300:24$&$14$&$12$&$11$&$3$&&Frobenius group\\
\stepcounter{case}$\thecase$&$D_{12}$&$300:25$&$14$&$6$&$3$&$5$&&\\
\stepcounter{case}$\thecase$&$D_{12}$&$600:59$&$11$&$3$&$2$&$5$&\checkmark&non-principal\\ 
\stepcounter{case}$\thecase$&$Z_3\rtimes Z_4$&$300:23$&$8$&$6$&$5$&$2$&&Frobenius group\\
\stepcounter{case}$\thecase$&$Z_4^2$&$400:205$&$25$&$16$&$9$&$3$&&$(Z_5\rtimes Z_4)^2$\\ 
\stepcounter{case}$\thecase$&$Z_4^2$&$800:957$&$13$&$4$&$1$&$3$&\checkmark&non-principal\\
\stepcounter{case}$\thecase$&$Z_4^2$&$1600:5606$&$10$&$1$&$0$&$3$&\checkmark&non-principal\\ 
\stepcounter{case}$\thecase$&$D_8*Z_4$&$400:207$&$16$&$10$&$6$&$3$&&\\
\stepcounter{case}$\thecase$&$D_8*Z_4$&$800:968$&$10$&$4$&$2$&$3$&\checkmark&non-principal\\
\stepcounter{case}$\thecase$&$M_{16}$&$400:206$&$13$&$10$&$8$&$3$&&\\
\stepcounter{case}$\thecase$&$Z_{24}$&$600:149$&$25$&$24$&$23$&$2$&&$\AGL(1,25)$\\ 
\stepcounter{case}$\thecase$&$\SL_2(3)$&$600:150$&$8$&$7$&$6$&$2$&&\\
\stepcounter{case}$\thecase$&$Z_4\times S_3$&$600:151$&$16$&$12$&$9$&$3$&&\\
\stepcounter{case}$\thecase$&$Z_4\times S_3$&$1200:491$&$10$&$6$&$5$&$3$&&non-principal\\ 
\stepcounter{case}$\thecase$&$Z_3\rtimes Z_8$&$600:148$&$13$&$12$&$11$&$2$&&Frobenius group\\
\stepcounter{case}$\thecase$&$Z_4\wr Z_2$&$800:1191$&$20$&$14$&$9$&$3$&&$(Z_5\rtimes Z_4)\wr Z_2$\\ 
\stepcounter{case}$\thecase$&$Z_4\wr Z_2$&$1600:9791$&$11$&$5$&$3$&$3$&\checkmark&non-principal\\ 
\stepcounter{case}$\thecase$&$\SL_2(3)*Z_4$&$1200:947$&$16$&$14$&$12$&$3$&&\\
\stepcounter{case}$\thecase$&$\GammaL_1(25)$&$1200:946$&$20$&$18$&$16$&$3$&&$\AGammaL(1,25)$\\
\stepcounter{case}$\thecase$&$\SL_2(3)\rtimes Z_4$&$-$&$20$&$16$&$12$&$3$&&$\mathtt{PrimitiveGroup}(25,19)$\\\bottomrule
\end{tabular} 
\end{center}
\end{small}
\end{Thm}
\begin{proof}
Most of the arguments work as in \autoref{local3}, but we have to be careful if the Schur multiplier of $E:=I(B)$ is larger than $Z_2$. 
For $E\cong Z_4^2$ the Schur multiplier is $Z_4$ by the Künneth formula. A full cover of $L:=D\rtimes E$ is given by $\widehat{L}:=\mathtt{SmallGroup}(1600,5606)$. This group has four blocks: the principal block, two faithful blocks and a non-faithful block. One can show by computer that $\widehat{L}$ has an automorphism acting as inversion on $\Z(\widehat{L})\cong Z_4$. It follows that the two faithful blocks are isomorphic. Hence, $\widehat{L}$ has only three types of blocks and they have pairwise distinct invariants. In this way we obtain the lines $25$, $26$ and $27$ in the table. 

For $E\cong D_8*Z_4$ the Schur multiplier is $Z_2^2$. A full cover of $L$ is given by $\widehat{L}:=\mathtt{SmallGroup}(1600, 5725)$. Fortunately, $\widehat{L}$ has an automorphism of order $3$ which permutes $\Z(\widehat{L})$. Hence, the three non-principal blocks of $\widehat{L}$ are all isomorphic and the Morita equivalence class of $B$ is uniquely determined in this case. This yields lines $28$ and $29$ in the table. The existence of an isotypy between $B$ and $B_D$ is an outcome of our algorithm.

Finally, we need to verify that all forty blocks are pairwise not Morita equivalent. Comparing the numerical invariants leads to the blocks no. $4$ and $23$. These can be distinguished using \cite{FrobeniusInertial} just as in \autoref{local3}. 
\end{proof}

The cases $1$, $10$ and $27$ on our list confirm a result of Kessar--Linckelmann~\cite{KLone} (for a recent generalization see \cite{HulB1}). 

An analysis of the Brauer trees mentioned above reveals that the Morita equivalence classes of blocks with defect group $Z_5$ are represented by the principal blocks of the groups $Z_5$, $D_{10}$, $Z_5\rtimes Z_4$, $A_5$, $S_5$ or $\Sz(8)$. Now taking direct products of two of these groups already yields $15$ Morita equivalence classes of principal blocks which do not belong to the classes in \autoref{local5}. For symmetric groups there are precisely 26 Morita equivalent classes of blocks with defect group $Z_5\times Z_5$ (see \cite{SambaleMorita}). We have found over 100 more classes by checking the character table library in GAP.

\section{The Cartan method revisited}

As we have seen in the last section, Usami and Puig's method fails in some situations.
We provide an alternative by improving the Cartan method described in \cite[Section~4.2]{habil}. This reduces the possible Cartan matrices of blocks to a handful of choices which can be discussed individually (most cases contradict Alperin's Weight Conjecture). As another advantage, the method applies equally well to non-abelian defect groups. To do so, the inertial quotient $I(B)$ must be replaced by the fusion system $\mathcal{F}$ of $B$. Nevertheless, the reader will notice many similarities to Usami--Puig's approach (in fact both methods can produce perfect isometries, see \cite[Theorem~6.1]{SambaleIso}).

The key idea is the orthogonality between the decomposition matrix $Q_1$ of $B$ and the generalized decomposition matrices $Q_u$ for $u\in D\setminus\{1\}$ (see \cite[Theorem~1.14]{habil}). We wish to compute $Q_u$ with Plesken's algorithm applied to the equation $Q_u^\mathrm{t}\overline{Q_u}=C_u$ where $C_u$ is the Cartan matrix of $b_u$ (as before, we obtain $C_u$ from the dominated block $\overline{b_u}$). To this end, we first need to “integralize” $Q_u$ by expressing its columns as linear combinations of an integral basis in a suitable cyclotomic field (in our situation we use the basis $1$, $\zeta=e^{2\pi i/3}$ of $\QQ_3$). We put these new integral columns in a “fake” generalized decomposition matrix $\widetilde{Q}_u$. Although $\widetilde{Q}_u$ has more columns than $Q_u$, both matrices generated the same orthogonal space. In practice we will remove linearly dependent columns from $\widetilde{Q}_u$ to obtain a matrices with $l(b_u)|\Aut(\langle u\rangle):\mathcal{N}_u|$ columns where $\mathcal{N}_u:=\N_G(\langle u\rangle,b_u)/\C_G(u)$ (this step is not strictly necessary).
The scalar products between the columns of $\widetilde{Q}_u$ can be computed by studying the action of $\mathcal{N}_u$ on $\IBr(b_u)$. This gives rise to the “fake” Cartan matrix $\widetilde{C}_u:=\widetilde{Q}_u^\mathrm{t}\widetilde{Q}_u$ (this was develop in general in \cite[Theorem~2.1]{Sambalerefine}, but in our case hand calculations will do). We obtain such a matrix for every $\mathcal{F}$-conjugacy class of cyclic subgroups of $D$. 
Unfortunately, $\widetilde{C}_u$ depends crucially on the chosen basic set for $b_u$. 
We introduce the following results to find “good” basic sets.

\begin{Lem}\label{autQF}
Let $C=(d+\delta_{ij})_{i,j=1}^n\in\ZZ^{n\times n}$ where $d$ and $n$ are positive integers. Let \[\Aut(C):=\{A\in\GL(n,\ZZ):A^\mathrm{t}CA=C\}.\] 
Then there exists a natural isomorphism
\[\Aut(C)\cong \begin{cases}
S_n\times Z_2&\text{if }d>1 \text{ or }n=1,\\
S_{n+1}\times Z_2&\text{if }d=1<n
\end{cases} \]
sending $A\in\Aut(C)$ to $\pm P$ where $P$ is a permutation matrix. 
\end{Lem}
\begin{proof}
We may assume that $n>1$. We first solve the matrix equation $(c_{ij})=C=Q^\text{t}Q$ where $Q\in\ZZ^{(n+d)\times n}$ has non-zero rows. To this end, we define the positive definite matrix
\[W:=(w_{ij})=\frac{1}{2}\begin{pmatrix}
2&-1&&0\\
-1&\ddots&\ddots\\
&\ddots&\ddots& -1\\
0&&-1&2
\end{pmatrix}\in\QQ^{n\times n}.\]
For the rows $q_1,\ldots,q_{n+d}$ of our putative solution $Q=(q_{ij})$ we obtain
\begin{align*}
n+d&\le\sum_{i=1}^{n+d}q_iWq_i^\mathrm{t}=\sum_{i=1}^{n+d}\sum_{1\le s,t\le n}w_{st}q_{is}q_{it}=\sum_{1\le s,t\le n}w_{st}c_{st}\\
&=\sum_{i=1}^nc_{ii}-\sum_{i=1}^{n-1}c_{i,i+1}=n(d+1)-(n-1)d=n+d.
\end{align*}
It follows that 
\[\frac{1}{2}\Bigl(q_{i1}^2+q_{in}^2+\sum_{j=1}^{n-1}(q_{ij}-q_{i,j+1})^2\Bigr)=\sum_{j=1}^nq_{ij}^2-\sum_{j=1}^{n-1}q_{ij}q_{i,j+1}=q_iWq_i^\mathrm{t}=1\] 
for $i=1,\ldots,n+d$. Hence, every row of $Q$ has the form $\pm(0,\ldots,0,1,\ldots,1,0,\ldots,0)$. Now it is easy to see that 
\[Q:=\begin{pmatrix}
1&&0\\
&\ddots\\
0&&1\\
-1&\cdots&-1\\
\vdots&&\vdots\\
-1&\cdots&-1
\end{pmatrix}\in\ZZ^{(n+d)\times n}\]
is the only solution of the equation $C=Q^\text{t}Q$ (up to permutations and signs of rows) of size $(n+d)\times n$. Hence, for $A\in\Aut(C)$ there exists a signed permutation matrix $P$ such that $QA=PQ$, since $(QA)^\mathrm{t}(QA)=A^\mathrm{t}CA=C$. Note that $A$ is just the upper part of $PQ$. At closer look at line $d+1$ reveals that $P$ has in fact a uniform sign, i.\,e. $P$ or $-P$ is a permutation matrix. The map $f:\Aut(C)\to S_{n+d}\times Z_2$, $A\mapsto P$ is clearly a monomorphism.

For $d=1$ the matrix $P$ has size $(n+1)\times(n+1)$ and conversely every such permutation matrix $P$ gives rise to some $A\in\Aut(C)$ such that $PQ=QA$, since the upper part of $PQ$ has determinant $\pm1$. Hence, $f$ is surjective in this case.
On the other hand, if $d>1$, then $P$ must fix the last $d$ rows of $Q$ and therefore $A$ or $-A$ itself must be a permutation matrix. Thus, in this case $\Aut(C)$ consists of the permutation matrices and their negatives.
\end{proof}

The next proposition generalizes \cite[Lemma~3.3]{Sambalerefine}.

\begin{Prop}\label{basicsets}
Let $B$ be a $p$-block of a finite group $G$ with abelian defect group $D$ such that $E:=I(B)$ is abelian and $D\rtimes E$ is a Frobenius group. Suppose that $p>2$ or $|E|<|D|-1$. Suppose further that $B$ is perfectly isometric to its Brauer correspondent $B_D$ in $\N_G(D)$. Let $\alpha\in\Aut(G)$ such that $\alpha(B)=B$. Then there exists a basic set $\Phi$ of $B$ such that $\IBr(B)$ and $\Phi$ are isomorphic $\alpha$-sets and the Cartan matrix of $B$ with respect to $\Phi$ is 
\[C=\Bigl(\frac{|D|-1}{|E|}+\delta_{ij}\Bigr)_{i,j=1}^{|E|}.\] 
\end{Prop}
\begin{proof}
It is well-known that the abelian Frobenius complement $E$ is in fact cyclic. Therefore, $E$ has trivial Schur multiplier and $B_D$ is Morita equivalent to the group algebra of the Frobenius group $L:=D\rtimes E$. The irreducible characters of $L$ are either inflations from $E$ or induced from $D$. 
The inflations from $E$ can be identified with the irreducible Brauer characters of $L$.
On the other hand, the number of distinct irreducible characters induced from $D$ is $d:=\frac{|D|-1}{|E|}$.
Consequently, the decomposition matrix of $B_D$ is
\[Q_D:=\begin{pmatrix}
1&&0\\
&\ddots\\
0&&1\\
1&\cdots&1\\
\vdots&&\vdots\\
1&\cdots&1
\end{pmatrix}\]
and the Cartan matrix is $Q_D^\mathrm{t}Q_D=C$ as given in the statement. 

Let $Q_B:=(d_{\chi\phi})$ and $C_B:=Q_B^\mathrm{t}Q_B=(c_{\phi\mu})$ be the decomposition matrix and the Cartan matrix of $B$ respectively.
Since $B$ is perfectly isometric to $B_D$, there exist $S\in\GL(l(B),\ZZ)$ and a signed permutation matrix $T\in\GL(k(B),\ZZ)$ such that $Q_DS=TQ_B$ (see \cite[Theorem~4.2]{SambaleIso}). Note that $Q_D$ differs from $Q$ in the proof of \autoref{autQF} only by the signs of the last rows. We replace $T$ by $T'$ accordingly such that 
\[QS=T'Q_B.\] 
After rearranging $\Irr(B)$, we may assume that $T'$ is just an identity matrix with signs. 

The action of $\alpha$ on $\IBr(B)$ permutes the columns of $Q_B$. Let $P$ be the corresponding permutation matrix. Since 
\[c_{\phi^\alpha,\mu^\alpha}=\sum_{\chi\in\Irr(B)}d_{\chi,\phi^\alpha}d_{\chi,\mu^\alpha}=\sum_{\chi\in\Irr(B)}d_{\chi^{\alpha^{-1}},\phi}d_{\chi^{\alpha^{-1}},\mu}=c_{\phi\mu}\]
for $\phi,\mu\in\IBr(B)$, it follows that $P$ commutes with $C_B$. We compute
\[(S^\mathrm{-t}P^\mathrm{t}S^{\mathrm{t}})C(SPS^{-1})=(S^\mathrm{-t}P^\mathrm{t}S^{\mathrm{t}})S^\mathrm{-t}C_BS^{-1}(SPS^{-1})=S^\mathrm{-t}P^\mathrm{t}C_BPS^{-1}=S^\mathrm{-t}C_BS^{-1}=C,\]
i.\,e. $A:=SPS^{-1}\in\Aut(C)$. By \autoref{autQF}, there exist a permutation matrix $P_A$ and a sign $\epsilon=\pm1$ such that 
\[QA=\epsilon P_AQ.\] 

Suppose first that $P_A$ has no fixed points. Then $d=1$ and neither $A$ nor $-A$ is a permutation matrix. It follows that $-\epsilon=\tr(A)=\tr(SPS^{-1})=\tr(P)\ge 0$ and $\epsilon=-1$. 
We compute
\[T'Q_BP=QSP=QAS=-P_AQS=-P_AT'Q_B\]
and $Q_BP=-T'P_AT'Q_B$. Now $Q_B$ and $Q_BP$ are non-negative matrices and $-T'P_AT'$ is a signed permutation matrix. This can only fit together if $-T'P_AT'=P_A$. In particular, 
\[\det P_A=\det(-T'P_AT')=(-1)^{|E|+d}\det P_A\] 
and $|D|=|E|+1=|E|+d$ is even. This contradicts the hypothesis $p>2$ (whenever $|E|=|D|-1$).

Now since $P_A$ has a fixed point, there exists yet another permutation matrix $U$ such that $UP_AU^{-1}$ fixes the last coordinate. We regard $UP_AU^{-1}$ as a permutation matrix $P'$ of size $n\times n$. Let $A_U\in\Aut(C)$ be the preimage of $U$ under the isomorphism in \autoref{autQF}, that is, $QA_U=UQ$. Then
\[QA_UAA_U^{-1}=UQAA_U^{-1}=\epsilon UP_AQA_U^{-1}=\epsilon UP_AU^{-1}Q.\]
By the shape of $Q$, it follows that $A_UAA_U^{-1}=\epsilon P'$. We may replace $S$ by $A_US$ if we adjust $T$ and $T'$ accordingly. Then we obtain $SPS^{-1}=\epsilon P'$. By way of contradiction, we assume that $\epsilon=-1$. Let $(a_1,\ldots,a_l)$ be the last row of $Q_B$.
Since $QS=T'Q_B$, we obtain
\[(a_1,\ldots,a_l)P=\pm(1,\ldots,1)SP=\mp(1,\ldots,1)P'S=\mp(1,\ldots,1)S=-(a_1,\ldots,a_l).\]
This is impossible since $a_1,\ldots,a_l\ge 0$ and at least one $a_i>0$. Hence, $\epsilon=1$.
A comparison of the eigenvalues shows that $P$ and $P'$ have the same cycle type. Consequently, $P$ and $P'$ are conjugate inside $S_n\le\Aut(C)$ as is well-known. Hence, we may change $S$, $T$ and $T'$ again such that $SPS^{-1}=P$. 

Finally, we define $\Phi=\{\widehat\phi_i:i=1,\ldots,l\}$ with $\widehat\phi_i:=\sum_{j=1}^ls_{ji}\phi_j$ where $\IBr(B)=\{\phi_1,\ldots,\phi_l\}$ and $S=(s_{ij})$. Then
\[(\widehat\phi_1^\alpha,\ldots,\widehat\phi_l^\alpha)=(\phi_1^\alpha,\ldots,\phi_l^\alpha)S=(\phi_1,\ldots,\phi_l)PS=(\phi_1,\ldots,\phi_l)SP=(\widehat\phi_1,\ldots,\widehat\phi_l)P,\]
i.\,e. $\Phi$ and $\IBr(B)$ are isomorphic $\alpha$-sets. Moreover, the decomposition matrix of $B$ with respect to $\Phi$ is $Q_BS^{-1}=Q$ and the Cartan matrix is $Q^\mathrm{t}Q=C$.
\end{proof}

The proof of \autoref{basicsets} does not go through for $p=2$ and $|E|=|D|-1$ as one can see from the possibility
\begin{align*}
Q_B=\begin{pmatrix}
1&.&.\\
.&1&.\\
1&.&1\\
.&1&1
\end{pmatrix},&&
S=\begin{pmatrix}
1&.&1\\
.&1&.\\
.&-1&-1
\end{pmatrix}
\end{align*}
with $\alpha$ being the transposition $(1,2)$ on $\IBr(B)$. In those exceptions $D$ is elementary abelian and $E$ is a Singer cycle. In a forthcoming paper of McKernon~\cite{McKernon} it will be shown that in this situation $B$ is Morita equivalent to the Brauer correspondent $B_D$ or to the principal block of $\SL(2,|D|)$. From the shape of the Cartan matrix of $\SL(2,|D|)$ (see \cite{AlperinSL}), it can be deduced that \autoref{basicsets} still holds in those cases. Hence, the hypothesis $p>2$ or $|E|<|D|-1$ is actually superfluous. 

Now we get back to the explanation of the Cartan method.
If $D$ is abelian, then all characters in $\Irr(B)$ have height $0$ (by \cite{KessarMalle}) and therefore every row of $Q_u$ (and of $\widetilde{Q}_u$) is non-zero. In general, the heights of the characters influence the $p$-adic valuation of the so-called \emph{contribution matrix}
\[M^u:=(m_{\chi\psi}^u)_{\chi,\psi\in\Irr(B)}=|D|Q_uC_u^{-1}Q_u^\mathrm{t}\in\CC^{k(B)\times k(B)}\]
(see \cite[Proposition~1.36]{habil}). 
This matrix is also of interest, because it only depends on the order of $\Irr(B)$ and possible signs, but not on the chosen basic set of $b_u$. In particular, there are at most $2^{k(B)}k(B)!$ choices for $M^u$, while there are potentially infinitely many choices for $Q_u$ (one for every basic set).
Note that
\[\widetilde{M}^u:=|D|\widetilde{Q}_u\widetilde{C}_u^{-1}\widetilde{Q}_u^\mathrm{t}=\sum_{\gamma\in\Aut(\langle u\rangle)/\mathcal{N}_u}M^{\gamma(u)}\in\ZZ^{k(B)\times k(B)},\]
since there exists an invertible complex matrix $U$ such that $(Q_{\gamma(u)}:\gamma\in\Aut(\langle u\rangle)/\mathcal{N}_u)=\widetilde{Q}_uU$.

Finally, Broué--Puig's $*$-construction, introduced in \cite{BrouePuigA}, gives congruence relations between the $M^u$ where $u$ runs through $D$. Specifically, if $\lambda$ is an $\mathcal{F}$-invariant generalized character of $D$, then \[\sum_{u\in\mathcal{S}}\lambda(u)M^u=|D|(\lambda*\chi,\psi)_{\chi,\psi\in\Irr(B)}\] 
where $\mathcal{S}$ is a set of representatives for the $\mathcal{F}$-conjugacy classes of elements of $D$. In particular, $\sum_{u\in\mathcal{S}}M^u=|D|1_{k(B)}$ and
\begin{equation}\label{BP}
\sum_{u\in\mathcal{S}}\lambda(u)M^u\equiv 0\pmod{|D|}.
\end{equation}

We summarize the steps of the Cartan method under the assumption that a defect group $D$ and a fusion system $\mathcal{F}$ on $D$ are given:
\begin{enumerate}[(1)]
\item Determine the $\mathcal{F}$-conjugacy classes of fully $\mathcal{F}$-centralized cyclic subgroups of $D$. Let $\mathcal{R}$ be a set of representatives of the corresponding generators.
\item For $u\in\mathcal{R}\setminus\{1\}$ determine the Cartan matrix $C_u$ (up to $E$-compatible basic sets) of a Brauer correspondent $b_u$ of $B$ in $\C_G(u)$ by considering the dominated block $\overline{b_u}$ with defect group $\C_D(u)/\langle u\rangle$ and fusion system $\C_{\mathcal{F}}(u)/\langle u\rangle$ (see \cite[Lemma~3]{SambaleC4}).
\item For every possible action of $\mathcal{N}_u$ on $\IBr(b_u)$ compute the fake Cartan matrix $\widetilde{C}_u$.
\item Solve the matrix equation $\widetilde{Q}_u^\mathrm{t}\widetilde{Q}_u=\widetilde{C}_u$ with Plesken's algorithm.
\item Reduce the number of possibilities for $\widetilde{Q}_u$ by comparing contribution matrices (make use of heights and the $*$-construction).
\item Form the matrix $\widetilde{Q}:=(\widetilde{Q}_u:u\in\mathcal{R}\setminus\{1\})$ of size $k(B)\times (k(B)-l(B))$. 
\item Compute an orthogonal complement $Q\in\ZZ^{k(B)\times l(B)}$ of $\widetilde{Q}$.
\item $C:=Q^\mathrm{t}Q$ is the Cartan matrix of $B$ up to basic sets. 
\end{enumerate}

\section{Cartan matrices of local blocks}

In this section we compute Cartan matrices of many $3$-blocks of defect at most $4$. They all occur as dominated Brauer correspondents of blocks with larger defect in the subsequent sections.

We first apply the Usami--Puig algorithm to the following situations.

\begin{Lem}\label{lem158}
Let $B$ be a block of a finite group $G$ with defect group $D\cong Z_3^3$ and inertial quotient $I(B)\cong D_8$ such that $D\rtimes I(B)\cong\mathtt{SmallGroup}(6^3,158)$. Then $B$ is perfectly isometric to its Brauer correspondent in $\N_G(D)$. In particular, $l(B)\in\{2,5\}$ and if $l(B)=5$, then the Cartan matrix of $B$ is given by 
\[
\begin{pmatrix}
7&5&2&1&6\\
5&7&1&2&6\\
2&1&7&5&6\\
1&2&5&7&6\\
6&6&6&6&15
\end{pmatrix}\]
up to basic sets.
\end{Lem}
\begin{proof}
The given group $L:=D\rtimes I(B)$ can be represented by $D=\langle x,y,z\rangle$ and $E:=I(B)=\langle a,b,c\rangle$ such that
\begin{align*}
^ax=x^{-1},&&^by=y^{-1},&&^cx=y,&&^cz=z^{-1},&&a^2=b^2=c^2=[a,y]=[b,x]=[a,z]=[b,z]=1.
\end{align*}
Since $E\cong D_8$ has Schur multiplier $Z_2$, there are two possible cocycles $\gamma$. For $\gamma\ne 1$ we can consider the (unique) non-principal block of $\widehat{L}=D\rtimes D_{16}\cong\mathtt{SmallGroup}(432,582)$.
The set $\mathcal{Q}$ in Step~(b) of the Usami--Puig algorithm consists only of those subgroups $Q<D$ which are normal in $L$, since otherwise $|\N_E(Q)|\le 4$. In all (three) cases we obtain the existence and uniqueness of the isometry $\Delta$ by Plesken's algorithm.
\end{proof}

\begin{Lem}\label{lem156}
Let $B$ be a block of a finite group $G$ with defect group $D\cong Z_3^3$ and inertial quotient $I(B)\cong Z_4\times Z_2$ such that $D\rtimes I(B)\cong\mathtt{SmallGroup}(6^3,156)$. Then $B$ is perfectly isometric to its Brauer correspondent in $\N_G(D)$. In particular, $l(B)\in\{2,8\}$ and if $l(B)=8$, then the Cartan matrix of $B$ is given by 
\[\begin{pmatrix}
6&2&3&4&2&2&4&4\\
2&6&4&3&4&4&2&2\\
3&4&6&2&4&4&2&2\\
4&3&2&6&2&2&4&4\\
2&4&4&2&6&4&3&2\\
2&4&4&2&4&6&2&3\\
4&2&2&4&3&2&6&4\\
4&2&2&4&2&3&4&6
\end{pmatrix}
\]
up to basic sets.
\end{Lem}
\begin{proof}
As in the previous lemma, $E:=I(B)=\langle a\rangle\times\langle b\rangle\cong Z_4\times Z_2$ acts reducibly on $D=\langle x,y,z\rangle$ such that $L:=D\rtimes I(B)=\langle x,y,a\rangle\times\langle z,b\rangle$.
Again $E$ has Schur multiplier $Z_2$ and there are two possible cocycles $\gamma$. For $\gamma\ne 1$ we can consider the non-principal block of $\mathtt{SmallGroup}(432,568)$.

The set $\mathcal{Q}$ in Step~(b) of the Usami--Puig algorithm consists of $1$, $\langle z\rangle$ and $\langle x,y\rangle$. 
Our algorithm works for $Q=1$ without intervention, but needs some additional argument for the remaining two cases. We only deal with $Q=\langle z\rangle$, since the final case is similar, in fact easier. 
The arguments go along the lines of \autoref{defect2}. We may assume that $G=\N_G(Q,b_Q)$.
Let $\beta_Q$ be a block of a suitable stem extension $\widehat{L}$ of $L$ such that $\beta_Q$ is isomorphic to $F_\gamma\overline{\C_L(Q)}$.
Note that $\beta_Q$ and $\overline{b_Q}$ have defect group $\langle x,y\rangle$ and inertial quotient $\C_E(Q)=\langle a\rangle\cong Z_4$. 
By \autoref{local3} we know that $l(\beta_Q)=l(\overline{b_Q})=4$. 
Since $E/\C_E(Q)=\langle b\rangle$ is cyclic, \autoref{basicsets} provides us with a basic set $\Phi$ of $\overline{b_Q}$ such that $\Phi$ and $\IBr(\overline{b_Q})$ are isomorphic $E$-sets and the Cartan matrix of $\overline{b_Q}$ with respect to $\Phi$ is $C=(2+\delta_{ij})_{i,j=1}^4$. This happens to be the Cartan matrix of $\beta_Q$. As in the proof of \autoref{defect2} we can extend $\Delta_0$ by any bijection $\PIM(\beta_Q)\to\{\widehat{\phi}:\phi\in\Phi\}$. It suffices to show that $\IBr(\beta_Q)$ and $\IBr(\overline{b_Q})$ are isomorphic $E$-sets.

Suppose first that $\gamma=1$. Then $E$ acts trivially on $\Irr(\beta_Q)$ and $\overline{L}\cong\overline{\C_L(Q)}\times\langle b\rangle$. Therefore, $\beta_Q$ is covered by two blocks of $\overline{L}$ and they both have inertial quotient $Z_4$. By \cite[Proposition~3.14]{UsamiZ2Z2}, $\overline{b_Q}$ is covered by two blocks of $\overline{G}$ with inertial quotient $Z_4$. According to Clifford theory, $E$ must act trivially on $\Irr(\overline{b_Q})$. 

Now suppose that $\gamma\ne1$. Then $\beta_Q$ is covered by only one block $\overline{B_Q}$ of $\widehat{L}/Q$ and $I(\overline{B_Q})\cong Z_2$. Hence, $l(\overline{B_Q})=2$ by \autoref{local3}. Clifford theory implies that $E$ acts as a double transposition on $\IBr(\beta_Q)$. 
Again by \cite[Proposition~3.14]{UsamiZ2Z2}, $\overline{b_Q}$ is covered by a unique block of $\overline{G}$ and this block has also two irreducible Brauer characters. Consequently, $E$ also acts as a double transposition on $\IBr(\overline{b_Q})$. Hence, in any event, $\IBr(\beta_Q)$ and $\IBr(\overline{b_Q})$ are isomorphic $E$-sets.
\end{proof}

In the following lemmas the Usami--Puig method fails basically because one of the two “bad” groups $Z_8$ and $Q_8$ from \autoref{local3} is involved in $I(B)$. We make use of the Cartan method instead.

\begin{Lem}\label{lem155}
Let $B$ be a block of a finite group $G$ with defect group $D\cong Z_3^3$ and inertial quotient $I(B)\in\{Z_8,Q_8\}$ such that $D\rtimes I(B)\cong\mathtt{SmallGroup}(6^3,s)$ where $s\in\{155,161\}$. Then $l(B)\in\{2,5,8\}$ and in the latter two cases the Cartan matrix of $B$ is 
\begin{align*}
\begin{pmatrix}
5&4&3&3&6\\
4&5&3&3&6\\
3&3&5&4&6\\
3&3&4&5&6\\
6&6&6&6&15
\end{pmatrix}
&&\text{or}&&
\begin{pmatrix}
5&4&3&3&3&3&3&3\\
4&5&3&3&3&3&3&3\\
3&3&5&4&3&3&3&3\\
3&3&4&5&3&3&3&3\\
3&3&3&3&5&4&3&3\\
3&3&3&3&4&5&3&3\\
3&3&3&3&3&3&5&4\\
3&3&3&3&3&3&4&5
\end{pmatrix}
\end{align*}
up to basic sets.
\end{Lem}
\begin{proof}
Let $D=\langle x,y,z\rangle$ and $E:=I(B)$. In both cases $E$ acts regularly on $\langle x,y\rangle$ and inverts $z$. We will see that the local analysis does not depend on the isomorphism type of $E$ and we will end up with exactly the same possibilities for the generalized decomposition matrices.
Let $L:=D\rtimes E$. Since the fusion system of $B$ is the fusion system of $L$ and every subgroup of $D$ is fully $\mathcal{F}$-centralized, we may choose $\mathcal{R}=\{1,x,xz,z\}$ in the algorithm of the Cartan method. Since $\C_E(x)=\C_E(xz)=1$ and $\C_E(z)\cong Z_4$ we obtain $l(b_x)=l(b_{xz})=1$ and $l(b_z)=4$ by applying \autoref{local3} to the dominated blocks $\overline{b_x}$ and so on. 
Since $x$ is conjugate to $x^{-1}$ in $L$, we have $Q_x=\widetilde{Q}_x$ and $C_x=\widetilde{C}_x=(27)$. 
On the other hand, $xz$ is not conjugate to $(xz)^{-1}$. We may write $Q_{xz}=(a+b\zeta)$ and form $\widetilde{Q}_{xz}=(a,b)$ with $\zeta=e^{2\pi i/3}$ and $a,b\in\ZZ^{k(B)\times 1}$. Since $C_{xz}=(27)$, we compute 
\[\widetilde{C}_{xz}=
\begin{pmatrix}1&1\\\zeta&\overline{\zeta}\end{pmatrix}^{-\mathrm{t}}
\begin{pmatrix}
27&0\\
0&27
\end{pmatrix}
\begin{pmatrix}1&1\\\overline{\zeta}&\zeta\end{pmatrix}^{-1}
=9\begin{pmatrix}
2&1\\1&2
\end{pmatrix}.\]

By \autoref{local3}, $\overline{b_z}$ is perfectly isometric to its Brauer first main theorem correspondent. Moreover, $\overline{D}\rtimes I(b_z)\cong Z_3^2\rtimes Z_4$ is a Frobenius group. Hence by \autoref{basicsets}, there exists a basic set $\Phi$ of $\overline{b_z}$ such that 
\[\mathcal{N}_z=\N_G(\langle z\rangle,b_z)/\C_G(z)\cong Z_2\]
acts on $\Phi$ and $\overline{b_z}$ has Cartan matrix $(2+\delta_{ij})_{i,j=1}^4$ with respect to $\Phi$. 
By \cite[Theorem~9.10]{Navarro}, we may regard $\Phi$ as a basic set of $b_z$ and then $C_z=3(2+\delta_{ij})_{i,j=1}^4$. Applying \cite[Theoem~4.2]{habil} to $b_z$ yields $k(B)\le 18$. Additionally, $k(B)$ is always divisible by $3$ (this can be seen from $Q_x$ or \cite[Proposition~1.31]{habil} in general). On the other hand,
\[k(B)-l(B)=l(b_x)+l(b_{xz})+l(b_{(xz)^{-1}})+l(b_z)=7.\]
Since we may assume that $l(B)>2$, we are left with the cases $(k(B),l(B))\in\{(18,11),(15,8),(12,5)\}$.

If $\rho$ is the regular character of $\langle x,y\rangle$, then 
\[\lambda:=(9\cdot 1_{\langle x,y\rangle}-\rho)\times 1_{\langle z\rangle}\]
is an $E$-invariant generalized character of $D$ such that $\lambda(1)=\lambda(z)=0$ and $\lambda(x)=\lambda(xz)=9$. We obtain $M^x+\widetilde{M}^{xz}\equiv 0\pmod{3}$ from \eqref{BP}.
Similarly, there exists an $E$-invariant $\lambda$ with $\lambda(1)=\lambda(x)=0$ and $\lambda(xz)=\lambda(z)=3$. This implies $\widetilde{M}^{xz}+M^z\equiv 0\pmod{9}$. Additionally, $M^1+M^x+\widetilde{M}^{xz}+M^z=27\cdot 1_{k(B)}$. In particular, 
\begin{equation}\label{contrib}
2m^{xz}_{\chi\chi}+m^z_{\chi\chi}\in\{9,18\} 
\end{equation}
for all $\chi\in\Irr(B)$.
We discuss the possible actions of $\mathcal{N}_z$ on $\Phi$. 

\textbf{Case~1:} $\mathcal{N}_z$ acts trivially on $\Phi$.\\
Here we have $Q_z=\widetilde{Q}_z$ and $C_z=\widetilde{C}_z$. We apply Plesken's algorithm directly to the block matrix $\widetilde{C}_{xz}\oplus C_z$ with a prescribed set of rows fulfilling \eqref{contrib}. Solutions exist only if $k(B)=15$. Taking also the congruence $\widetilde{M}^{xz}+M^z\equiv 0\pmod{9}$ into account, there is a unique solution up to basic sets:
\[(\widetilde{Q}_{xy},Q_z)=\begin{pmatrix}
1&-2&1&1&1&1\\
-2&1&1&1&1&1\\
-1&-1&2&2&2&2\\
1&1&1&.&.&.\\
1&1&1&.&.&.\\
1&1&1&.&.&.\\
1&1&.&1&.&.\\
1&1&.&1&.&.\\
1&1&.&1&.&.\\
1&1&.&.&1&.\\
1&1&.&.&1&.\\
1&1&.&.&1&.\\
1&1&.&.&.&1\\
1&1&.&.&.&1\\
1&1&.&.&.&1
\end{pmatrix}\]
Now it is easy to add the column $Q_x$ under the condition $M^x+\widetilde{M}^{xz}\equiv 0\pmod{3}$. In the end, the Cartan matrix $C$ of $B$ is uniquely determined up to basic sets.

\textbf{Case~2:} $\mathcal{N}_z$ interchanges two characters of $\Phi$.\\
We may assume that the first two characters of $\Phi$ are interchanged by $\mathcal{N}_z$. 
We express the first columns of $Q_z$ with the basis $1,\zeta$ and compute
\[\widetilde{C}_z=
\begin{pmatrix}
1&1&.&.\\
\zeta&\overline{\zeta}&.&.\\
.&.&1&.\\
.&.&.&1
\end{pmatrix}^{-\mathrm{t}}
C_z
\begin{pmatrix}
1&1&.&.\\
\overline{\zeta}&\zeta&.&.\\
.&.&1&.\\
.&.&.&1
\end{pmatrix}^{-1}=
\begin{pmatrix}
8&1&6&6\\
1&2&0&0\\
6&0&9&6\\
6&0&6&9
\end{pmatrix}
.\]
It is convenient to reduce this matrix with the LLL algorithm to
\[\begin{pmatrix}
2&1&1&1\\
1&5&2&2\\
1&2&5&2\\
1&2&2&8
\end{pmatrix}\]
(this amounts a change of basic sets).
Plesken's algorithm applied to $\widetilde{C}_{xz}\oplus \widetilde{C}_z$ under the restriction \eqref{contrib} yields $k(B)=12$. 
As in Case~1, there is in fact a unique solution:
\[(\widetilde{Q}_{xy},\widetilde{Q}_z)=\begin{pmatrix}
1&1&1&1&1&1\\
2&2&.&1&1&1\\
-2&1&.&.&.&1\\
1&-2&.&.&.&1\\
-1&-1&1&.&.&.\\
-1&-1&.&.&.&2\\
-1&-1&.&1&.&.\\
-1&-1&.&1&.&.\\
-1&-1&.&1&.&.\\
-1&-1&.&.&1&.\\
-1&-1&.&.&1&.\\
-1&-1&.&.&1&.
\end{pmatrix}
\]
Combined with the possibilities for $Q_x$ one gets the Cartan matrix of $B$ up to basic sets.

\textbf{Case~3}: $\mathcal{N}_z$ has two orbits of length $2$ on $\Phi$.\\
Let $(1,2)(3,4)$ be the cycle structure of $\mathcal{N}_z$ on $\Phi$.
Then
\[\widetilde{C}_z=
\begin{pmatrix}
1&1&.&.\\
\zeta&\overline{\zeta}&.&.\\
.&.&1&1\\
.&.&\zeta&\overline{\zeta}
\end{pmatrix}^{-\mathrm{t}}
C_z
\begin{pmatrix}
1&1&.&.\\
\overline{\zeta}&\zeta&.&.\\
.&.&1&1\\
.&.&\overline{\zeta}&\zeta
\end{pmatrix}^{-1}=
\begin{pmatrix}
8&1&6&0\\
1&2&0&0\\
6&0&8&1\\
0&0&1&2
\end{pmatrix}\sim_{LLL}\begin{pmatrix}
2&.&1&1\\
.&2&1&.\\
1&1&4&2\\
1&.&2&8
\end{pmatrix}.
\]
It turns out that Plesken's algorithm for $\widetilde{C}_{xz}\oplus\widetilde{C}_z$ only has solutions if $k(B)=9$. This case was already excluded.
\end{proof}

With the notation of the proof above, we note that Case~1 occurs if $I(B)\cong Z_8$ and Case~2 occurs if $I(B)\cong Q_8$. Case~3 contradicts Alperin's Weight Conjecture (cf. remark after \autoref{local3}).

We also need an extension of the defect group in \autoref{lem155}.

\begin{Lem}\label{lem155x}
Let $B$ be a block of a finite group $G$ with defect group $D\cong Z_3^4$ and inertial quotient $I(B)\cong Z_8$ such that $D\rtimes I(B)\cong\mathtt{SmallGroup}(6^3,155)\times Z_3$. Then the Cartan matrix of $B$ is $3C$ where $C$ is one of the possible Cartan matrices in \autoref{lem155}.
\end{Lem}
\begin{proof}
The proof follows along the lines of \cite[Proposition~16]{SambaleC4}. Let $E:=I(B)$. We note that $D=D_1\times D_2$ with $D_1:=[D,E]=\langle x,y,z\rangle\cong Z_3^3$ and $D_2:=\C_E(D)=\langle w\rangle\cong Z_3$. With the notation of \autoref{lem155}, a set of representatives for the $E$-orbits on $D$ is given by $\mathcal{R}=\{w^i,xw^i,zw^i,xzw^i:i=0,1,2\}$. The character group $\Irr(D_2)$ acts semiregularly on $\Irr(B)$ via the $*$-construction. By \cite[Lemma~10]{SambaleC4}, the generalized decomposition matrix $Q_{uw^i}$ (where $u\in\{1,x,z,xz\}$) has the form
\[Q_{uw^i}=\begin{pmatrix}
A_{uw^i}\\
\zeta^i A_{uw^i}\\
\overline{\zeta}^i A_{uw^i}
\end{pmatrix}.\]
Now the ordinary decomposition matrix $Q_1$ is orthogonal to $Q_r$ for all $r\in\mathcal{R}\setminus\{1\}$ if and only if $A_1$ is orthogonal to $A_x$, $A_{z}$ and $A_{xz}$. Just as in \autoref{lem155} we have $I(b_x)=I(b_{xz})=1$ and $I(b_z)\cong Z_4$. Therefore the corresponding Cartan matrices $C_x$, $C_z$ and $C_{xz}$ are given by \cite{UsamiZ4}. By the structure of the matrix above, we obtain $3A_x^\mathrm{t}\overline{A_x}=C_x$ and similarly for $z$ and $xz$. Consequently, the matrices $A_x$, $A_z$ and $A_{xz}$ fulfill the very same relations as the matrices $Q_x$, $Q_z$ and $Q_{xz}$ in the proof of \autoref{lem155}. In particular, we obtain exactly the same possibilities for $A_1$. From that we compute $C=3A_1^\mathrm{t}A_1$.
\end{proof}

\begin{Lem}\label{lem160}
Let $B$ be a block of a finite group with defect group $D\cong Z_3^3$ and $I(B)\cong Q_8$ such that $D\rtimes I(B)\cong M_9\times Z_3$. Then $l(B)\in\{2,5,8\}$ and in the latter two cases the Cartan matrix of $B$ is 
\begin{align*}
3\begin{pmatrix}
2&1&1&1&2\\
1&2&1&1&2\\
1&1&2&1&2\\
1&1&1&2&2\\
2&2&2&2&5
\end{pmatrix}
&&\text{or}&&
3(1+\delta_{ij})_{i,j=1}^8
\end{align*}
up to basic sets.
\end{Lem}
\begin{proof}
Let $D:=\langle x,y,z\rangle$ such that $E:=I(B)\cong Q_8$ acts regularly on $\langle x,y\rangle$ and $\C_D(E)=\langle z\rangle$.
The proof is similar as the previous one. With the notation used there, $A_1$ is just the orthogonal complement of $A_x$. 
Observe that $l(b_x)=1$ and $x$ is conjugate to $x^{-1}$ under $E$. So there are essentially three choices for $A_x$:
\begin{align*}
(2,2,1)^\mathrm{t},&&(2,1,1,1,1)^\mathrm{t},&&(1,\ldots,1)^\mathrm{t}.
\end{align*}
The claim follows as usual. 
\end{proof}

For our next local block we need to establish $E$-compatible basic sets of virtually non-existent blocks.

\begin{Lem}\label{bsQ8}
Let $B$ be a block of finite group $G$ with defect group $D\cong Z_3^2$ such that $D\rtimes I(B)\cong M_9$. 
Suppose that $B$ is fixed by some automorphism $\alpha\in\Aut(G)$ of order $2$. 
Then there exists a basic set $\Phi$ such that $\Phi$ and $\IBr(B)$ are isomorphic $\alpha$-sets and the Cartan matrix of $B$ with respect to $\Phi$ is one of the following
\begin{align*}
\begin{pmatrix}
5&4\\4&5
\end{pmatrix},&&
\begin{pmatrix}
2&1&1&1&2\\
1&2&1&1&2\\
1&1&2&1&2\\
1&1&1&2&2\\
2&2&2&2&5
\end{pmatrix},&&
(1+\delta_{ij})_{i,j=1}^8.
\end{align*}
\end{Lem}
\begin{proof}
Since $I(B)$ acts regularly on $D$, there is only one non-trivial subsection $(x,b_x)$ and $l(b_x)=1$. 
As in the previous lemma, the generalized decomposition matrix $Q_x$ is one of the following: $(2,2,1)^\mathrm{t}$, $(2,1,\ldots,1)^\mathrm{t}$ or $(1,\ldots,1)^\mathrm{t}$ up to signs. From that we obtain the decomposition matrix $Q$ of $B$ up to basic sets. More precisely, there exists some $S\in\GL(l(B),\ZZ)$ and an identity matrix with signs $T$ such that $TQS$ is one of the following
\begin{align*}
\begin{pmatrix}
1&.\\
.&1\\
2&2
\end{pmatrix},&&
\begin{pmatrix}
1&.&.&.&.\\
.&1&.&.&.\\
.&.&1&.&.\\
.&.&.&1&.\\
.&.&.&.&1\\
1&1&1&1&2
\end{pmatrix},&&
\begin{pmatrix}
1&&0\\
&\ddots\\
0&&1\\
1&\cdots&1
\end{pmatrix}\in\ZZ^{9\times 8}.
\end{align*}
If $\alpha$ acts trivially on $\IBr(B)$, then we choose $\Phi$ according to $S$. The Cartan matrix with respect to $\Phi$ is then given by $\widehat{C}:=S^\mathrm{t}Q^\mathrm{t}QS$ as in the statement. 
Thus, we may assume that $\alpha$ acts non-trivially on $\IBr(B)$. Suppose first that $l(B)=2$. Since $\alpha$ interchanges the two Brauer characters of $B$, the Cartan matrix of $B$ has the form $C=Q^\mathrm{t}Q=\bigl(\begin{smallmatrix}s&t\\t&s
\end{smallmatrix}\bigr)$. Since $C$ has the same elementary divisors as $\widehat{C}=\bigl(\begin{smallmatrix}
5&4\\4&5
\end{smallmatrix}\bigr)$, we conclude that $9=\det C=s^2-t^2=(s+t)(s-t)$. This easily implies $C=\widehat{C}$. Hence, we may choose $\Phi=\IBr(B)$. 

Suppose next that $l(B)=5$. Since $\alpha$ has order $2$, we may arrange $\IBr(B)=\{\phi_1,\ldots,\phi_5\}$ such that $\alpha$ fixes $\phi_5$.
Let $P$ be the permutation matrix on the columns of $Q$ induced by $\alpha$. Then 
\[(S^{\mathrm{t}}P^\mathrm{t}S^{-\mathrm{t}})\widehat{C}(S^{-1}PS)=\widehat{C},\]
i.\,e. $S^{-1}PS\in\Aut(\widehat{C})$. 
One can show by computer (or as in \autoref{autQF}) that $\Aut(\widehat{C})\cong S_5\times Z_2$ where $Z_2$ is generated by the negative identity matrix and $S_5$ contains the permutation matrices on the first four coordinates. In particular, $P\in\Aut(\widehat{C})$.
Since $P$ and $S^{-1}PS$ have the same rational canonical form, the computer tells us that $P$ and $S^{-1}PS$ are conjugate inside $\Aut(\widehat{C})$. Hence, we may assume that $PS=SP$. Now the claim follows as in the proof of \autoref{basicsets}. 

Now let $l(B)=8$. Then $\Aut(\widehat{C})$ was already computed in \autoref{autQF} and we can repeat the arguments in \autoref{basicsets} word by word.
\end{proof}

\begin{Lem}\label{lem432}
Let $B$ be a block of a finite group with defect group $D\cong Z_3^3$ and $I(B)\cong Q_8\times Z_2$ such that $D\rtimes I(B)\cong M_9\times S_3$. Let $C$ be the Cartan matrix of $B$. Then there exists a matrix $W\in\RR^{l(B)\times l(B)}$ such that $xWx^\mathrm{t}\ge 1$ for all $x\in\ZZ^{l(B)}\setminus\{0\}$ and $\tr(WC)\le 27$. 
\end{Lem}
\begin{proof}
Let $D=\langle x,y,z\rangle$ and $E:=I(B)=\langle a,b,c\rangle$ such that 
\[L:=D\rtimes E=\langle x,y,a,b\rangle\times\langle z,c\rangle\cong M_9\times S_3.\]
The $B$-subsections are represented by $\mathcal{R}=\{1,x,z,xz\}$ and all of these elements are conjugate to their inverses under $L$. As usual, we obtain $l(b_x)=2$, $l(b_{xz})=1$ and $l(b_z)\in\{2,5,8\}$.
Moreover, $\widetilde{Q}_{xz}=Q_{xz}$ and $\widetilde{C}_{xz}=C_{xz}=(27)$. 
If $\mathcal{N}_x=\N_G(\langle x\rangle,b_x)/\C_G(x)\cong\langle a^2\rangle$ acts non-trivially on $\IBr(b_x)$, then the \emph{exact} Cartan matrix of $b_x$ is $C_x=9\bigl(\begin{smallmatrix}
2&1\\1&2
\end{smallmatrix}\bigr)$ as can be seen from the elementary divisors just as in the proof of \autoref{bsQ8}.
If, on the other hand, $\mathcal{N}_x$ acts trivially on $\IBr(b_x)$ we choose a basic set such that $C_x=9\bigl(\begin{smallmatrix}
2&1\\1&2
\end{smallmatrix}\bigr)$
So there are two possibilities
\begin{align*}
\widetilde{C}_x=C_x=9\begin{pmatrix}
2&1\\1&2
\end{pmatrix}&&\text{or}&&\widetilde{C}_x=\begin{pmatrix}
1&1\\\zeta&\overline{\zeta}
\end{pmatrix}^{-\mathrm{t}}C_x\begin{pmatrix}
1&1\\\overline{\zeta}&\zeta
\end{pmatrix}^{-1}=3\begin{pmatrix}
5&1\\1&2
\end{pmatrix}.
\end{align*}
which we denote by (A) and (B). 
The block $\overline{b_z}$ fulfills the hypothesis of \autoref{bsQ8}. We choose a corresponding basic set $\Phi$ and observe that $\mathcal{N}_z=\N_G(\langle z\rangle,b_z)/\C_G(z)\cong \langle c\rangle$ has cycle type $(1)$, $(2)$, $(2^2)$, $(2^3)$, $(2^4)$ provided $l(b_z)$ is large enough. The relevant fake Cartan matrices $\widetilde{C}_z$ are
\begin{gather*}
3\begin{pmatrix}
5&4\\4&5
\end{pmatrix},
\begin{pmatrix}
14&1\\1&2
\end{pmatrix},
3\begin{pmatrix}
2&1&1&1&2\\
1&2&1&1&2\\
1&1&2&1&2\\
1&1&1&2&2\\
2&2&2&2&5
\end{pmatrix},
\begin{pmatrix}
5&1&3&3&6\\
1&2&.&.&.\\
3&.&6&3&6\\
3&.&3&6&6\\
6&.&6&6&15
\end{pmatrix},
\begin{pmatrix}
5&1&3&.&6\\
1&2&.&.&.\\
3&.&5&1&6\\
.&.&1&2&.\\
6&.&6&.&15
\end{pmatrix}\\
3(1+\delta_{ij})_{i,j=1}^8,
\begin{pmatrix}
5&1&3&3&3&3&3&3\\
1&2&.&.&.&.&.&.\\
3&.&6&3&3&3&3&3\\
3&.&3&6&3&3&3&3\\
3&.&3&3&6&3&3&3\\
3&.&3&3&3&6&3&3\\
3&.&3&3&3&3&6&3\\
3&.&3&3&3&3&3&6
\end{pmatrix},
\begin{pmatrix}
5&1&3&.&3&3&3&3\\
1&2&.&.&.&.&.&.\\
3&.&5&1&3&3&3&3\\
.&.&1&2&.&.&.&.\\
3&.&3&.&6&3&3&3\\
3&.&3&.&3&6&3&3\\
3&.&3&.&3&3&6&3\\
3&.&3&.&3&3&3&6
\end{pmatrix},\\
\begin{pmatrix}
5&1&3&.&3&.&3&3\\
1&2&.&.&.&.&.&.\\
3&.&5&1&3&.&3&3\\
.&.&1&2&.&.&.&.\\
3&.&3&.&5&1&3&3\\
.&.&.&.&1&2&.&.\\
3&.&3&.&3&.&6&3\\
3&.&3&.&3&.&3&6
\end{pmatrix},
\begin{pmatrix}
5&1&3&.&3&.&3&.\\
1&2&.&.&.&.&.&.\\
3&.&5&1&3&.&3&.\\
.&.&1&2&.&.&.&.\\
3&.&3&.&5&1&3&.\\
.&.&.&.&1&2&.&.\\
3&.&3&.&3&.&5&1\\
.&.&.&.&.&.&1&2
\end{pmatrix}.
\end{gather*}
We number these cases from (1) to (10). In total there are 20 cases to consider. 
As usual, 
\[k(B)-l(B)=l(b_x)+l(b_z)+l(b_{xz})\in\{5,8,11\}.\] 
Note that $\widetilde{M}^x=M^x$ and $\widetilde{M}^z=M^z$.
As in the proof of \autoref{lem155} we obtain $M^x+M^{xz}\equiv 0\pmod{3}$ and $M^z+M^{xz}\equiv 0\pmod{9}$ from \eqref{BP}. 
It is remarkable that Plesken's algorithm applied to $C_{xz}\oplus \widetilde{C}_z$ under the condition $M^z+M^{xz}\equiv 0\pmod{9}$ always has a unique solution up to basic sets. In particular, $k(B)$ and $l(B)$ are uniquely determined in each of the cases (1)--(10). If $k(B)$ is given, there are only a few solutions $\widetilde{Q}_x$ for $\widetilde{Q}_x^\mathrm{t}\widetilde{Q}_x=\widetilde{C}_x$. Eventually we combine $(Q_{xz},\widetilde{Q}_z)$ with $\widetilde{Q}_x$. 
We collect the possible pairs $(k(B),l(B))$ in the following table:
\[\begin{array}{c|cccccccccc}
\text{Case}&(1)&(2)&(3)&(4)&(5)&(6)&(7)&(8)&(9)&(10)\\\hline
(A)&(9,4)&(6,1)&(18,10)&-&(12,4)&(27,16)&-&(21,10)&(18,7)&(15,4)\\
(B)&-&-&-&(15,7)&(12,4)&-&-&-&(18,7)&(15,4)
\end{array}\]
In each case we obtain a handful of possible Cartan matrices $C$ of $B$. If 
\[m:=\min\bigl\{xC^{-1}x^\mathrm{t}:x\in\ZZ^{l(B)}\setminus\{0\}\bigr\}\ge \frac{l(B)}{27},\]
then the claim follows with $W:=\frac{1}{m}C^{-1}$. This works in all cases except Case~(8A). For this exception we construct $W$ from an identity matrix by adding some entries $\pm\frac{1}{2}$ matching the positions of the “large” off-diagonal entries of $C$. Then we check if $W$ is positive definite (cf. \cite[proof of Theorem~13.7]{habil}). This can all be done in an automatic fashion.
Finally, we remark that the existence of $W$ does not depend on the basic set for $B$ (see \cite[Introduction]{SambaleBound}).
\end{proof}

Now we turn to non-abelian defect groups. For the definitions of controlled and constrained fusion systems (and blocks) we refer to \cite[p. 11]{habil}.

\begin{Lem}\label{lemextra}
Let $B$ be a controlled block with extraspecial defect group $D\cong 3^{1+2}_+$ of order $27$ and exponent $3$. Suppose that $I(B)\cong Z_2$ and $D\rtimes I(B)=\mathtt{SmallGroup}(54,5)$. Then $l(B)\le 2$.
\end{Lem}
\begin{proof}
Let $D=\langle x,y,z\rangle$ and $E:=I(B)=\langle a\rangle$ such that
\begin{align*}
[x,y]=z,&& x^3=y^3=z^3=[x,z]=[y,z]=1,&&x^a=x^{-1},&&y^a=y,&&z^a=z^{-1}.
\end{align*}
Since the fusion system $\mathcal{F}$ of $B$ is controlled, every subgroup of $D$ is fully $\mathcal{F}$-centralized. Hence, for $u\in D$ the Brauer correspondent $b_u$ of $B$ in $\C_G(u)$ has defect group $\C_D(u)$ and $I(b_u)\cong\C_E(u)$. 
It can be checked that the $E$-orbits on $D$ are represented by the elements $1,z,x,y,y^{-1},xy,xy^{-1}$. 
Since $\C_E(z)=\C_E(x)=\C_E(xy)=\C_E(xy^{-1})=1$ and $\C_E(y)=E$, we have $l(b_z)=l(b_x)=l(b_{xy})=l(b_{xy^{-1}})=1$ and $l(b_y)=l(b_{y^{-1}})=2$. It follows that 
\[k(B)-l(B)=4\cdot 1+2\cdot 2=8.\] 
By \cite[Theorem~4.2]{habil} applied to $b_x$, we obtain $k_0(B)\le 9$ where $k_h(B)$ denotes the number characters of height $h$  in $B$. Hence, we may assume that $l(B)\ge 3$ and $k_1(B)\ge 2$. Now \cite[Theorem~4.7]{habil} applied to $b_z$ yields $k(B)=k_0(B)+k_1(B)=9+2=11$ and $l(B)=3$. Consequently, the generalized decomposition numbers corresponding to $(z,b_z)$ have the form $(\pm1,\ldots,\pm1,\pm3,\pm3)^\mathrm{t}$ where the first nine entries correspond to the characters of height $0$. Similarly, the generalized decomposition numbers with respect to $(x,b_x)$ are $(\pm1,\ldots,\pm1,0,0)^\mathrm{t}$. But now these two columns cannot be orthogonal (regardless of the signs). This contradicts \cite[Theorem~1.14]{habil}.
Hence, $l(B)\le 2$.
\end{proof}

In the following the famous group
\[\Qd(3)=\ASL(2,3)=Z_3^2\rtimes\SL(2,3)\cong\mathtt{SmallGroup}(6^3,153)\]
plays a role.

\begin{Lem}\label{lemextra2}
Let $B$ be a constrained block with extraspecial defect group $D\cong 3^{1+2}_+$ and fusion system $\mathcal{F}=\mathcal{F}(\Qd(3))$. Then $l(B)\le2$ or the Cartan matrix of $B$ is
\[C=\begin{pmatrix}
4&1&2\\
1&4&2\\
2&2&7
\end{pmatrix}\]
up to basic sets.
\end{Lem}
\begin{proof}
Let $D=\langle x,y,z\rangle$ as in \autoref{lemextra} and $E=\langle a,b\rangle\cong Q_8$ such that $E$ acts on $\langle x,z\rangle$ (but not on $D$). Then $I(B)=\langle a^2\rangle=\langle b^2\rangle$ and $x^{a^2}=x^{-1}$, $y^{a^2}=y$ and $z^{a^2}=z^{-1}$.
The $\mathcal{F}$-conjugacy classes of fully centralized cyclic subgroups of $D$ are represented by $\mathcal{R}=\{1,z,y,xy\}$.
We compute $l(b_z)=l(b_{xy})=1$ and $l(b_y)=2$. 
Since $z$ is $\mathcal{F}$-conjugate to $z^{-1}$, we obtain $Q_z=\widetilde{Q}_z$ and $C_z=\widetilde{C}_z=(27)$. On the other hand, $xy$ and $(xy)^{-1}$ are not $\mathcal{F}$-conjugate. From $C_{xy}=(9)$ we conclude
\[\widetilde{C}_{xy}=\begin{pmatrix}1&1\\\zeta&\overline{\zeta}\end{pmatrix}^{-\mathrm{t}}
\begin{pmatrix}
9&0\\
0&9
\end{pmatrix}
\begin{pmatrix}1&1\\\overline{\zeta}&\zeta\end{pmatrix}^{-1}=3\begin{pmatrix}
2&1\\1&2
\end{pmatrix}.\]
Similarly, $y$ and $y^{-1}$ are not $\mathcal{F}$-conjugate. It follows from Brauer's theory of blocks of defect $1$ (see \cite[Theorem~11.4]{Navarro}) that 
$\bigl(\begin{smallmatrix}
2&1\\1&2
\end{smallmatrix}\bigr)$ is the \emph{exact} Cartan matrix of $\overline{b_y}$ (no just up to basic sets). Therefore, $C_y=3\bigl(\begin{smallmatrix}
2&1\\1&2
\end{smallmatrix}\bigr)$ is the Cartan matrix of $b_y$ and we may assume that $\overline{Q_y}=Q_{y^{-1}}$. Thus, we do not need to find $E$-compatible basic sets as in \autoref{basicsets}.
We compute
\[\widetilde{C}_y=
\begin{pmatrix}
1&1&.&.\\
\zeta&\overline{\zeta}&.&.\\
.&.&1&1\\
.&.&\zeta&\overline{\zeta}
\end{pmatrix}^{-\mathrm{t}}
\begin{pmatrix}
6&.&3&.\\
.&6&.&3\\
3&.&6&.\\
.&3&.&6
\end{pmatrix}
\begin{pmatrix}
1&1&.&.\\
\overline{\zeta}&\zeta&.&.\\
.&.&1&1\\
.&.&\overline{\zeta}&\zeta
\end{pmatrix}^{-1}
=
\begin{pmatrix}
4&2&2&1\\
2&4&1&2\\
2&1&4&2\\
1&2&2&4
\end{pmatrix}.\]
Moreover, 
\[k(B)-l(B)=l(b_z)+l(b_y)+k(b_{y^{-1}})+l(b_{xy})+l(b_{(xy)^{-1}})=7.\] 
We may assume that $l(B)>2$ and therefore, $k(B)\ge 10$.
By \cite[Theorem~4.2]{habil} applied to $b_{xy}$, it follows that $k_0(B)\le 9$. Hence, $k_0(B)\in\{3,6,9\}$ by \cite[Proposition~1.31]{habil}.
Now \cite[Proposition~4.7]{habil} applied to $b_z$ yields $k_0(B)=9$ and $k_1(B)\le 2$. In total, $k(B)\le 11$.
Recall that the height zero characters correspond to non-zero rows in $\widetilde{Q}_y$ and $\widetilde{Q}_{xy}$. By Plesken's algorithm, the non-zero parts $\widetilde{Q}_y^0$ and $\widetilde{Q}_{xy}^0$ of these matrices are essentially unique:
\begin{align*}
\widetilde{Q}_y^0=\begin{pmatrix}
1&1&1&1\\
1&1&.&.\\
1&.&1&.\\
1&.&.&.\\
.&1&.&1\\
.&1&.&.\\
.&.&1&1\\
.&.&1&.\\
.&.&.&1
\end{pmatrix},&&
\widetilde{Q}_{xy}^0=\begin{pmatrix}
1&.\\
1&.\\
1&.\\
.&1\\
.&1\\
.&1\\
1&1\\
1&1\\
1&1
\end{pmatrix}
\end{align*}
The $*$-construction with suitable $\mathcal{F}$-invariant characters implies 
\[3M^z\equiv\widetilde{M}^y\equiv-\widetilde{M}^{xy}\pmod{9}\]
by \eqref{BP}.
Under these restrictions $\widetilde{Q}_y$ and $\widetilde{Q}_{xy}$ can only be combined in a few ways. 

Suppose that $k(B)=11$. Then $Q_z=(\pm1,\ldots,\pm1,\pm3,\pm3)^\mathrm{t}$ where the first nine characters have height $0$. However, one can show with GAP that $Q_z$ cannot be orthogonal to the matrix $(\widetilde{Q}_y,\widetilde{Q}_{xy})$ formed above. Consequently, $k(B)=10$ and $l(B)=3$. Then $Q_z=(\pm2,\pm2,\pm2,\pm1,\ldots,\pm1,\pm3)^\mathrm{t}$ up to permutations. All possible combinations lead to the desired Cartan matrix.
\end{proof}

We remark that the case $l(B)=3$ in \autoref{lemextra2} occurs and is predicted in general by Alperin's Weight Conjecture.

\section{Abelian defect groups}

We are now in a position to prove the first main result of this paper.

\begin{Thm}\label{abel}
Let $B$ be a $3$-block of a finite group $G$ with abelian defect group $D$ of rank at most $5$. Then $k(B)\le |D|$. 
\end{Thm}
\begin{proof}
Let $E:=I(B)$ as usual.
We decompose $D$ into indecomposable $E$-invariant subgroups 
\[D=D_1\times\ldots\times D_n\] 
where $n\le 5$ since $D$ has rank at most $5$. Since $E$ is a $3'$-automorphism group of $D$, we know that each $D_i$ is homocyclic, i.\,e. a direct product of isomorphic cyclic groups (see \cite[Theorem~5.2.2]{Gorenstein}). 
If $D_i$ is not elementary abelian or $|D_i|=p$, then there always exists $x_i\in D_i$ such that $\C_E(x_i)=\C_E(D_i)$ by \cite[Proposition~19]{SambaleC4}. If $|D_i|\le 27$, then there exists $x_i\in D_i$ such that $|\C_E(x_i):\C_E(D_i)|\le 2$. Hence, if $n\ge 3$, the element $x:=x_1\ldots x_n$ satisfies $|\C_E(x)|\le 4$. In this case the claim follows from \cite[Lemma~14.5]{habil}. 

Now suppose that $n=2$. 
Then we may assume that $D_1\cong Z_3^4$. Let $x_2$ be a generator of the cyclic group $D_2$. 
We may assume that there is no $x_1\in D_1$ such that $|\C_E(x_1):\C_E(D_1)|\le 4$, because otherwise $|\C_E(x_1x_2)|\le 4$. 
The action of $E$ on $D_1$ determines an irreducible $3'$-subgroup $\overline{E}:=E/\C_E(D_1)$ of $\GL(4,3)$. 
In other words, $L_1:=D_1\rtimes \overline{E}$ is a primitive permutation group on $D_1$ of affine type. These groups are fully classified and available in GAP.
It turns out that there are three possibilities:
\begin{enumerate}[(i)]
\item $\overline{E}\cong SD_{32}\wr Z_2$ (a Sylow $2$-subgroup of $\GL(4,3)$) and $L_1\cong\mathtt{PrimitiveGroup}(3^4,95)$,
\item $\overline{E}\cong \mathtt{SmallGroup}(2^8,6662)$ and $L_1\cong\mathtt{PrimitiveGroup}(3^4,83)$,
\item $\overline{E}\cong \mathtt{SmallGroup}(640, 21454)$ and $L_1\cong\mathtt{PrimitiveGroup}(3^4,99)$.
\end{enumerate}
In the first two cases there exists $x_1\in D_1$ such that $\C_{\overline{E}}(x_1)\cong D_8$. With $x:=x_1x_2$ we obtain $\C_E(x)\cong D_8$. Let $b_x$ be the corresponding Brauer correspondent of $B$ in $\C_G(x)$ and let $\overline{b_x}$ be the dominated block in $\C_G(x)/\langle x\rangle$ with defect group $\overline{D}:=D/\langle x\rangle$.  
Another GAP computation shows that $\overline{D}\rtimes I(\overline{b_x})\cong\mathtt{SmallGroup}(6^3,158)\times D_2/\langle x_2^3\rangle$ where the second factor has order $3$ if $D_2\ne 1$. Our Usami--Puig algorithm applied in \autoref{lem158} works equally well if $D_2\ne1$ (only Plesken takes a little longer). Consequently, $l(b_x)=l(\overline{b_x})\in\{2,5\}$. 
If $l(b_x)\le 2$, then the claim follows from \cite[Theorem~4.9]{habil}. Otherwise we know the Cartan matrix $\overline{C_x}$ of $\overline{b_x}$ up to basic sets from \autoref{lem158} (for $D_2\ne1$ the given matrix must be multiplied by $3$). The Cartan matrix of $b_x$ is $3\overline{C_x}$ (up to basic sets). Now the claim follows from \cite[Theorem~4.2]{habil}.

In the third case above we find $x\in D$ such that $\overline{D}\rtimes I(\overline{b_x})\cong\mathtt{SmallGroup}(6^3,155)\times D_2/\langle x_2^3\rangle$. This time we get the Cartan matrix of $b_x$ from Lemmas~\ref{lem155} and \ref{lem155x} (again assuming $l(b_x)>2$). Here the claim follows from \cite[Theorem~4.4]{habil} (the minimum of the quadratic form can be computed with the GAP command \texttt{ShortestVectors}).

Finally, it remains to handle $n=1$, i.\,e. $E$ acts irreducibly on $D$. By \cite[Proposition~11]{SambaleC4}, we may assume that $|\C_E(x)|>7$ for all $x\in D$. From the GAP library we see that the only primitive group to consider is $L=D\rtimes E\cong \mathtt{PrimitiveGroup}(3^5,15)$. Here we find $x\in D$ such that $\overline{D}\rtimes I(\overline{b_x})\cong\mathtt{SmallGroup}(6^3,158)\times Z_3$. This is the same instance already discussed above.
\end{proof}

The inertial index $256$ occurring in Case~(ii) of the above proof is in fact the smallest inertial index where Brauer's $k(B)$-Conjecture is not known to hold in general (see \cite[Proposition~14.13]{habil}). 

\section{Non-abelian defect groups}

In order to investigate non-abelian defect groups, we first generalize \cite[Theorem~8]{SambaleC4}.

\begin{Thm}\label{meta}
Let $B$ be a block of a finite group with defect group $D$ such that $D/\langle z\rangle$ is metacyclic for some $z\in\Z(D)$. Then $k(B)\le|D|$.
\end{Thm}
\begin{proof}
If $p=2$ or $\overline{D}:=D/\langle z\rangle$ is abelian, then the claim follows from \cite[Theorem~13.9]{habil} or \cite[Theorem~5]{SambaleC4} respectively. Thus, we may assume that $p>2$ and $\overline{D}$ is non-abelian. Let $b_z$ be a Brauer correspondent of $B$ in $\C_G(z)$. As usual, $b_z$ dominates a unique block $\overline{b_z}$ of $\overline{C_z}:=\C_G(z)/\langle z\rangle$ with defect group $\overline{D}$. In order to apply \cite[Theorem~4.2]{habil}, we need to compute the Cartan matrix of $\overline{b_z}$ up to basic sets. To this end, we may assume that $\overline{b_z}$ is non-nilpotent. Then by \cite[Theorem~8.8]{habil}, 
\[\overline{D}=\langle x,y:x^{p^m}=y^{p^n}=1,\ yxy^{-1}=x^{1+p^l}\rangle\]
where $0<l<m$ and $m-l\le n$. By a result of Stancu~\cite{Stancu}, $\overline{b_z}$ is a controlled block. Moreover, $E:=I(\overline{b_z})$ is cyclic of order dividing $p-1$ and $E$ acts semiregularly on $\langle x\rangle$ and trivially on $\langle y\rangle$ (see \cite[proof of Theorem~8.8]{habil}). In particular, the (hyper)focal subgroup $[\overline{D},E]=\langle x\rangle$ is cyclic. 
By the main result of \cite{TasakaWatanabe}, $\overline{b_z}$ is perfectly isometric to its Brauer correspondent $\beta_z$ in $\N_{\overline{C_z}}(\overline{D})$. In particular, $\overline{b_z}$ and $\beta_z$ have the same Cartan matrices up to basic sets.
By \cite{Kuelshammer}, we may assume that $\beta_z$ is the unique block of $L:=\overline{D}\rtimes E\cong C_{p^m}\rtimes(C_{p^n}\times E)$. 
By result of Fong (see \cite[Theorem~10.13]{Navarro}), the projective indecomposable characters of $L_1:=\langle x\rangle\rtimes E$ are $\Phi_\lambda':=\lambda^{L_1}$ where $\lambda\in\Irr(E)$. Similarly, the projective indecomposable characters of $L$ are $\Phi_\lambda:=\lambda^L=(\Phi_\lambda')^L$. Since $\langle y\rangle$ centralizes $E$, we have $\Phi_\lambda(g)=p^n(\Phi_\lambda')(g)$ if $g\in L_1$ and $0$ otherwise. Consequently,
\[[\Phi_\lambda,\Phi_\mu]=\frac{1}{|L|}\sum_{g\in L_1}p^{2n}\Phi_\lambda'(g)\overline{\Phi_\mu'(g)}=p^n[\Phi_\lambda',\Phi_\mu']\]
for $\lambda,\mu\in\Irr(E)$. The Cartan matrix of $L_1$ is $(d+\delta_{ij})_{i,j=1}^e$ where $e:=|E|$ and $d:=(p^m-1)/e$ (see proof of \autoref{basicsets}). Hence, the Cartan matrix of $L$ is $p^n(d+\delta_{ij})$ and the Cartan matrix of $b_z$ is $|\langle z\rangle|p^n(d+\delta_{ij})$. 
Now \cite[Theorem~4.2]{habil} yields
\[k(B)\le |\langle z\rangle|p^n\Bigl(\frac{p^m-1}{e}+e\Bigr)\le|\langle z\rangle|p^{n+m}=|\langle z\rangle||\overline{D}|=|D|.\qedhere\]
\end{proof}

We can now prove our second main theorem.

\begin{Thm}\label{nonabel}
Brauer's $k(B)$-Conjecture holds for the $3$-blocks of defect at most $4$.
\end{Thm}
\begin{proof}
Brauer's Conjecture has been verified for all $p$-blocks of defect at most $3$ in \cite{SambaleRank3}. Hence, let $B$ be a block with defect group $D$ of order $81$.
By Theorems~\ref{abel} and \ref{meta} we may assume that $D$ is non-abelian and $D/\langle z\rangle$ has order $27$ and exponent $3$ for every $z\in\Z(D)\setminus\{1\}$. 

\textbf{Case~1:} $|\Z(D)|=3$.\\
Since there exists no extraspecial group of order $3^4$, we must have $D/\Z(D)\cong 3^{1+2}_+$. 
Using GAP we are left with four possible groups: $D\cong\mathtt{SmallGroup}(81,s)$ with $s=7,\ldots,10$.
Let $z\in\Z(D)\setminus\{1\}$, and let $b_z$ and $\overline{b_z}$ as usual.
The possible fusion systems of $\overline{b_z}$ were classified in \cite{ExtraspecialExpp}. 

We compute further that the $3'$-part of $|\C_{\Aut(D)}(\Z(D))|$ is at most $2$. In particular, $|I(\overline{b_z})|\le 2$. 
If $I(\overline{b_z})=1$, then $\overline{b_z}$ is nilpotent by the main theorem of \cite{ExtraspecialExpp} (this happens if $s=10$). In this case the claim follows from \cite[Propositon~4.7]{habil}.
Hence, we may assume that $|I(\overline{b_z})|=2$ in the following. A further calculation shows that 
$\overline{D}\rtimes I(\overline{b_z})\cong Z_3^2\rtimes Z_2\cong\mathtt{SmallGroup}(54,5)$. 

By \cite{ExtraspecialExpp}, the fusion system $\mathcal{F}_z$ of $\overline{b_z}$ is constrained. More precisely, $\mathcal{F}_z$ is the fusion system of the group $\overline{D}\rtimes I(\overline{b_z})$ or of the group $\Qd(3)$.
In the first case, we obtain $l(b_z)\le 2$ from \autoref{lemextra}. Then Brauer's $k(B)$-Conjecture follows from \cite[Theorem~9.4]{habil}. In the remaining case, the claim follows from \autoref{lemextra2} and \cite[Theorem~4.2]{habil}.

\textbf{Case~2:} $|\Z(D)|=9$.\\
Here $D\cong Z_3\times 3^{1+2}_+$. Let $z\in D'\setminus\{1\}$ and $b_z$, $\overline{b_z}$ as usual. Note that $\overline{b_z}$ has defect group $\overline{D}=D/D'=D/\Phi(D)\cong Z_3^3$. The $3'$-group $E:=I(B)\le\Aut(D)$ acts faithfully on $\overline{D}$ and normalizes $\Z(D)/D'$. Hence, $E$ is a $2$-group and $I(b_z)\cong\C_E(z)\le Q_8\times Z_2$. If $|I(b_z)|\le 4$, then the claim follows from \cite[Lemma~14.5]{habil}. Now suppose that $|I(b_z)|\in\{8,16\}$. Up to isomorphism there are four possibilities for $L:=\overline{D}\rtimes I(b_z)$:
\begin{enumerate}[(i)]
\item $I(b_z)\cong Z_4\times Z_2$ and $L\cong (Z_3^2\rtimes Z_4)\times S_3\cong\mathtt{SmallGroup}(6^3,156)$: Here the claim follows from \autoref{lem156} and \cite[Theorem~4.2]{habil}.

\item $I(b_z)\cong Q_8$ and $L\cong M_9\times Z_3$: Apply \cite[Theorem~4.2]{habil} with \autoref{lem160}.

\item $I(b_z)\cong Q_8$ and $L\cong\mathtt{SmallGroup}(6^3,161)$: Apply \cite[Theorem~4.2]{habil} with \autoref{lem155}.

\item $I(b_z)\cong Q_8\times Z_2$ and $L\cong M_9\times S_3$: In this case we use \cite[Theorem~A]{SambaleBound} in combination with \autoref{lem432}.\qedhere 
\end{enumerate}
\end{proof}

\section*{Acknowledgment}
Parts of this work were conducted while the first author visited the University of Hannover in November 2019. He appreciates the hospitality received there.
The authors thank Thomas Breuer for providing a refined implementation of Plesken's algorithm in GAP and Ruwen Hollenbach for stimulating discussing on the subject. 
The second author is supported by the German Research Foundation (\mbox{SA 2864/1-2} and \mbox{SA 2864/3-1}).

\begin{small}

\end{small}

\end{document}